\documentclass{amsart}
\usepackage[margin=1in]{geometry}
\usepackage{verbatim}
\usepackage[textsize=scriptsize]{todonotes}

\usepackage{tikz}
\usetikzlibrary{matrix,arrows}
\usepackage{tikz-cd}
\usepackage{amsmath}
\usepackage{relsize}
\usepackage{amssymb}
\usepackage{amsthm}
\usepackage{amscd}
\usepackage{enumerate}
\usepackage[pdfusetitle,unicode,hidelinks]{hyperref}
\usepackage{bbm}
\usepackage{etoolbox}
\usepackage{mathtools}
\usepackage{chngcntr}

\usepackage[utf8]{inputenc}

\newtoggle{final}
\toggletrue{final}

\numberwithin{equation}{section}

\newtheorem{proposition}[equation]{Proposition}
\newtheorem{corollary}[equation]{Corollary}
\newtheorem{lemma}[equation]{Lemma}
\newtheorem{theorem}[equation]{Theorem}
\newtheorem*{theorem*}{Theorem}
\newtheorem*{corollary*}{Corollary}
\newtheorem*{proposition*}{Proposition}
\newtheorem*{lemma*}{Lemma}
\theoremstyle{definition}
\newtheorem{definition}[equation]{Definition}
\newtheorem{construction}[equation]{Construction}

\newtheorem*{definition*}{Definition}
\newtheorem*{construction*}{Construction}
\newtheorem{variant}[equation]{Variant}
\theoremstyle{remark}
\newtheorem{remark}[equation]{Remark}
\newtheorem{example}[equation]{Example}
\newtheorem{warning}[equation]{Warning}

\newcommand{\triv}{\mathrm{triv}}

\newcommand{\id}{\operatorname{id}}
\newcommand{\Z}{\mathbb{Z}}
\def\C{\mathbb C}

\let\scr=\mathcal
\let\bb=\mathbb

\def\A{\bb A}

\def\P{\bb P}
\def\DD{\bb D}

\newcommand{\1}{\mathbbm{1}}
\newcommand{\hh}{\mathbbm{h}}
\newcommand{\Ex}{ex}

\newcommand{\DM}{\mathcal{DM}}
\newcommand{\SH}{\mathcal{SH}}

\def\ph{\mathord-}

\DeclareMathOperator*{\colim}{colim}

\let\lim=\relax
\DeclareMathOperator*{\lim}{lim}
\def\Map{\mathrm{Map}}

\def\stab{\mathrm{stab}}
\def\Nis{\mathrm{Nis}}
\def\Zar{\mathrm{Zar}}
\def\Stk{\mathrm{Stk}}

\def\NAlg{\mathrm{NAlg}}

\def\Sect{\mathrm{Sect}}
\def\PSh{\mathcal{P}}
\def\Shv{\mathcal{S}\mathrm{hv}}

\def\Span{\mathrm{Span}}

\def\Cat{\mathcal{C}\mathrm{at}{}}
\def\Spc{\mathcal{S}\mathrm{pc}{}}

\def\Fun{\mathrm{Fun}}

\def\triv{\mathrm{triv}}

\newcommand{\wequi}{\simeq}
\newcommand{\Mod}{\mathcal{M}od}
\def\adj{\rightleftarrows}

\DeclareRobustCommand{\ul}{\underline}

\newcommand{\tr}{{tr}}

\def\op{\mathrm{op}}

\let\cat=\mathrm
\def\Sm{{\cat{S}\mathrm{m}}}
\def\Aff{{\cat{A}\mathrm{ff}}}
\def\SmAff{{\cat{S}\mathrm{mAff}}}
\def\Sch{\cat{S}\mathrm{ch}{}}

\def\SmQP{\mathrm{SmQP}{}}
\def\FEt{\mathrm{FEt}{}}
\def\fet{\mathrm{f\acute et}}
\def\all{\mathrm{all}}
\def\et{\mathrm{\acute et}}
\def\pt{*}
\def\MGL{\mathrm{MGL}}

\def\mot{\mathrm{mot}}
\def\naive{\mathrm{naive}}

\def\EE{\mathbb{E}}

\def\D{\mathrm{D}}
\def\H{\mathrm{H}}
\def\BB{\mathbb{B}}

\def\NSym{\mathrm{NSym}}
\newcommand{\sslash}{\mathbin{/\mkern-6mu/}}

\iftoggle{final} {
\newcommand{\NB}[1]{}
\newcommand{\tombubble}[1]{}
\newcommand{\tom}[1]{}
\newcommand{\elden}[1]{}
\newcommand{\eldenbubble}[1]{}
\newcommand{\jeremiah}[1]{}
\newcommand{\sjeremiah}[1]{}
\renewcommand{\todo}[1]{}
}{ 
\newcommand{\NB}[1]{\todo[color=gray!40]{#1}}
\newcommand{\tombubble}[1]{\todo[color=green!40]{#1}}
\newcommand{\tom}[1]{{\color{green!60!black}#1}}
\newcommand{\elden}[1]{{\color{blue!60!black}#1}}
\newcommand{\eldenbubble}[1]{\todo[color=blue!40]{#1}}
\newcommand{\jeremiah}[1]{{\color{red!60!black}#1}}
\newcommand{\sjeremiah}[1]{{\todo[color=red!40]{#1}}}
}

\author{Tom Bachmann}
\address{Mathematisches Institut, LMU Munich, Munich, Germany}
\email{tom.bachmann@zoho.com}

\author{Elden Elmanto}
\address{Department of Mathematics, Harvard University, Cambridge, MA, USA}
\email{eldenelmanto@gmail.com}

\author{Jeremiah Heller}
\address{Department of Mathematics, University of Illinois, Urbana-Champaign, IL, USA}
\email{jbheller@illinois.edu}

\date{\today}

\newenvironment{subappendices}{\appendix}{}

\usepackage{xstring}
\usepackage{xr}
\newcommand{\myexternaldocument}[2][] {{
\let\nl\newlabel
\def\nlxx##1##2##3##4##5##6{\nl{##1}{{#1##2}{}{}{}{}}}
\renewcommand\newlabel[2]{\IfBeginWith{##1}{tocindent}{}{\nlxx{##1}##2}}
\externaldocument{#2}
}}

\title{Motivic colimits and extended powers}
\begin{document}
\maketitle

\begin{abstract}
We define a notion of colimit for diagrams in a motivic category indexed by a presheaf of spaces (e.g. an étale classifying space), and we study basic properties of this construction. As a case study,  we construct the motivic analogs of the classical extended and generalized powers, which refine the categorical versions of these constructions as special cases. We also offer more computationally tractable models of these constructions using equivariant motivic homotopy theory. This is the first in a series of papers on power operations in motivic stable homotopy theory.
\end{abstract}

\tableofcontents


\newcommand{\qproj}{\mathrm{qproj}}

\section{Introduction}
\subsection{Overview}
This is the first in a series of papers which introduce and study power operations in the context of motivic stable homotopy theory.
In this paper, we introduce a new construction which puts the theory on solid footing: motivic colimits. We demonstrate how this theory works by constructing a motivic version of extended and generalized powers.
This construction can be seen as generalizing certain ideas on \emph{motivic Thom spectra} by the first author and Hoyois \cite[\S16]{norms}, or as a motivic analog of \emph{equivariant colimits} in the sense of Shah \cite{shah-colimits}. It can also be interpreted as an instance of Lurie's extensive theory of \emph{relative colimits} \cite[\S4.3.1]{HTT}.

To describe this construction, let us fix a base scheme $S$. We will work with the category $\Sm_S$ of smooth $S$-schemes, which serve as the building blocks for motivic homotopy theory  over $S$. Roughly speaking the formalism of motivic colimits takes as input a \emph{motivic diagram}, which is a morphism of presheaves of $\infty$-categories on smooth $S$-schemes
\[
f: \mathcal{X} \rightarrow \SH(\ph),
\]
where $\mathcal{X}$ is a presheaf of $\infty$-groupoids. The output, the motivic colimit of $f$, is a motivic spectrum over the base scheme
\[
\colim_{\mathcal{X}} f \in \SH(S). 
\]

We give a flavor of the theory via some examples.

\begin{example} \label{ex:1} Suppose that $X$ is a smooth $S$-scheme with structure map $p_X: X \rightarrow S$. Via the Yoneda embedding, we can regard $X$ as a discrete presheaf on $\Sm_S$ which we denote by $h_X$. A motivic spectrum $E \in \SH(X)$ is classified by a morphism $E: h_X \rightarrow \SH(\ph)$, which is an example of a motivic diagram as above. The motivic colimit is then given by
\[
\colim_{h_X} E \simeq p_{X\sharp}E.
\] 
Here $p_{X\sharp}$ is the left adjoint to $p_X^*:\SH(S) \rightarrow \SH(X)$. In particular, if $h_X \rightarrow \SH$ classifies the relative sphere spectrum $\1_X \in \SH(X)$ then the motivic colimit is the suspension spectrum $\Sigma^{\infty}_+X \in \SH(S)$.
For another example, if $X = S \amalg S$ (so that $p$ is the fold map), then $E \in \SH(S \coprod S) \wequi \SH(S) \times \SH(S)$ really corresponds to two spectra $E_1, E_2 \in \SH(S)$, and $\colim_{h_X} E = E_1 \amalg E_2$ recovers an ordinary colimit.
\end{example}

\begin{example} \label{ex:2} Let $K$ be the presheaf of spaces sending $X$ to the algebraic $K$-theory space $K(X)$. As in \cite[\S16.2]{norms}, we have a canonical morphism of presheaves
\[
J:K \rightarrow \SH(\ph) \qquad \xi \mapsto S^{\xi}.
\]
Here $S^{\xi}$ is the Thom spectrum of the virtual vector bundle $\xi$.
In particular if $\xi$ was the class of an actual vector bundle $\mathcal{E}$, then
\[
S^{\xi} \simeq \Sigma^{\infty}\frac{\mathcal{E}}{\mathcal{E} \setminus 0}. 
\]
This morphism was christened the motivic $J$-homomorphism in \cite{norms}. Out of this we can extract two motivic colimits determining familiar objects:
\begin{enumerate}
\item precomposing with the inclusion of the rank zero part of $K$-theory $e:K^{\circ} \rightarrow K$, we get a motivic diagram
\[
J \circ e: K^{\circ} \rightarrow \SH(\ph),
\]
whose motivic colimit gives us Voevodsky's algebraic cobordism spectrum \cite[Theorem 16.13]{norms}
\[
\MGL_S \in \SH(S).
\]
\item Simply taking the motivic colimit of the $J$ gives us the periodized version of algebraic cobordism \cite[Remark 16.14]{norms}
\[
\bigvee_{j \in \Z} \Sigma^{2j,j}\MGL_S \in \SH(S). 
\]
\end{enumerate}
\end{example}

\begin{example} \label{ex:3} Let $\FEt$ denote the presheaf that sends an $S$-scheme $X$ to the groupoid of finite \'etale schemes over $X$ and fix $E \in \SH(S)$. We can construct a motivic diagram
\[
N_E: \FEt \rightarrow \SH(\ph),
\]
which, on a smooth $S$-scheme $Y$ with structure map $p_Y:Y \rightarrow S$, sends a finite \'etale $Y$-scheme $f:T \rightarrow Y$ to the motivic spectrum 
\[
f_{\otimes}(f\circ p_Y)^*E \in \SH(Y). 
\]
The resulting motivic colimit is the free normed spectrum on $E$ as introduced in \cite[Section 16.4]{norms}. 
\end{example}

\subsection{Why motivic colimits?} In \cite{norms}, the first author and Hoyois introduced the formalism of motivic Thom spectra in \cite[\S16]{norms} in order to construct normed structures on various cobordism spectra that appear in motivic homotopy theory; see \cite[Theorem 16.19, Example 16.22]{norms}. Such a structure packages an enormous amount of coherence data that would have been extremely difficult to write down by hand. Presenting these cobordism spectra as motivic Thom spectra allow the authors to pin-point the exact source of the coherence datum for the norm structure on these cobordism spectra, namely the functoriality of the motivic $J$-homomorphism; see \cite[Proposition 16.17]{norms} for a precise statement.

The formalism of motivic colimits is a generalization of the formalism of motivic Thom spectra. From the point-of-view of power operations in motivic homotopy theory, the formalism of motivic colimits allows us to produce a robust theory of motivic extended and generalized powers which, just as their classical counterparts, govern the power operations in highly structured ring spectra; see \cite[Chapter I]{hinfty} for an introduction to how this works in the classical setting.

Classically, if $E$ is a spectrum, then its $n$-th extended power is given by the formula
\[
D_n(E):=\colim_{B\Sigma_n} E^{\wedge n};
\]
 see \cite[Construction 2.2.1]{DAGXIII} for a reference in the language of this paper. Power operations on a highly structured ring spectrum $E$ are then governed by maps 
 \[
 D_n(E) \rightarrow E
 \] and their interactions as $n$ varies. Analogously, our construction of power operations in motivic homotopy theory requires the construction of such maps. 
 
 While the above construction makes sense in any presentably symmetric monoidal $\infty$-category, such as $\SH(S)$ (as explained in \emph{loc. cit.}), these \emph{categorical extended powers} will not produce ``genuinely motivic" power operations which are characterized by their ``weight-shifting" property. Indeed, one of the main theorems in the sequel to this paper is a motivic analog of Steinberger's deep result \cite[Chapter III, Theorems 2.2-2.3]{hinfty} which presents the dual Steenrod algebra as monogenically generated over the Dyer--Lashof algebra. One can build a kind of ``naive" motivic Dyer--Lashof algebra out of the categorical extended powers, but it will not be enough to generate the motivic Steenrod algebra. The fix, as is done in this paper, is to replace the colimit, taken over the groupoid $B\Sigma_n$, with a \emph{motivic colimit} over the étale classifying space of $\Sigma_n$ \cite[Section 4]{A1-homotopy-theory}, i.e. the classifying stack of $\Sigma_n$-torsors.

%
%

\subsection{Motivic colimits in practice}

In this paper, we will offer two demonstrations of motivic colimits in practice. First, we formulate a notion of \emph{motivic Kan extensions}. This will be used to prove a motivic analog of the following statement: if $\scr C$ is a semiadditive $\infty$-category, then each object $c \in \scr C$ admits a canonical map transfer map
\[
c \rightarrow c \oplus \cdots \oplus c,
\]
such that its composite with the fold map is given by mutiplication by $n$. This construction itself will be useful in our investigation of power operations. 

Second, we construct motivic extended and generalized powers.  This is an elaboration of Example~\ref{ex:3} which we briefly explain. For a base scheme $S$, we denote by $\FEt$ the classifying stack of finite \'etale morphisms thought of as a presheaf on $\Sm_S$. We will functorially pair the following two objects
\begin{enumerate}
\item a presheaf on smooth $S$-schemes $\scr X$ equipped with a map $\alpha: \scr X \rightarrow \FEt$, and
\item a motivic spectrum $E \in \SH(S)$,
\end{enumerate}
to construct a new motivic spectrum 
\[
D_\scr{X}^\mot(E) \in \SH(S),
\] 
which deserves to be called the \emph{motivic generalized power}. For example if $\scr X = B_{\et}\Sigma_n$ then $D_{\scr X}^{\mot}E$ deserves to be called the $n$-th \emph{motivic extended power}, which we denote by
\[
D_n^{\mot}(E).
\] On the other hand, if we let $\scr X = B\Sigma_n$ be the constant presheaf then this construction yields the the $n$-th categorical extended powers. The flexibility of choosing the map $\scr X \rightarrow B_{\et}\Sigma_n$, such as the map $B_{\et}H \rightarrow B_{\et}\Sigma_n$ induced by a subgroup inclusion $H \subset \Sigma_n$, will turn out to be useful, especially in identifying the various relations that occur within the motivic Dyer--Lashof algebra.

\subsection{Equivariant motivic homotopy theory} The ``genuine" nature of our construction of motivic extended powers is most easily explained via equivariant motivic homotopy theory. Two pieces of this latter theory play prominent roles. 

First, recall from \cite[Section 4.2]{A1-homotopy-theory} that via the construction of Morel--Voevodsky and Totaro, one can associate to a group scheme $G$ a motivic homotopy type $\BB G$ which approximates the \'etale classifying stack of $G$. It is via this construction that Voevodsky was able to construct power operations \cite{voevodsky2003reduced} which exhaust the motivic Steenrod algebra. The motivic-equivariant nature of this construction is explained in \cite[Section 3.1]{gepner-heller}. Let $\SH^{G}(S)$ be the stable $\infty$-category of motivic $G$-spectra (in the sense of \cite{Hoyois:6functor}; see also \cite[Section 2]{gepner-heller} and Appendix \ref{app:equiv-mot}). Then we have the motivic homotopy orbit \cite[Definition 7.3]{gepner-heller} functor which is defined on the subcategory of $\SH^{G}(S)$ generated by those schemes with free $G$-action (i.e. trivial isotropy groups)
\[
(\ph)/G:\SH^{G}(S)[\scr F_\triv] \rightarrow \SH(S).
\]
There is a motivic $G$-space $\EE G$ whose points classify schemes with free $G$-action \cite[Definition 3.2]{gepner-heller} and its suspension spectrum is an object of $\SH^G(S)[\scr F_\triv]$, while $\Sigma^\infty_+ \EE G/G$ recovers the suspension spectrum of $\BB G$. Smashing with $\EE G$ yields a colocalization  $\SH^{G}(S) \rightarrow \SH^{G}(S)[\scr F_\triv]$ and we define the composite
\[
(\ph)_{\hh G}:\SH^{G}(S) \rightarrow \SH^{G}(S)[\scr F_\triv] \rightarrow \SH(S),
\]
the \emph{geometric homotopy orbits functor}; see~\ref{subsec:quotients} for details.

Secondly, as we will show in this paper, the endofunctor that sends a motivic spectrum $E$ to its $n$-th power, $E^{\wedge n}$, promotes to a functor landing in the $\Sigma_n$-equivariant stable motivic category
\[
\SH(S) \rightarrow \SH^{\Sigma_n}(S) \qquad E \mapsto E^{\wedge \underline{n}}.
\]
This is a motivic analog of the \emph{norm construction} of Hill, Hopkins and Ravenel \cite{hhr} for the map $\ast \rightarrow BG$ induced by the inclusion of the identity element $\{e\} \subset G$.

With this, we obtain the following equivariant model for the motivic extended powers:

\begin{theorem}[Corollary~\ref{cor:d-vs-d}] Let $S$ be a scheme. There is a canonical equivalence:
\[
D_n^{\mot}(E) \simeq (E^{\wedge \underline{n}})_{\hh\Sigma_n}
\]
\end{theorem}

An application of this comparison comes in the form of Theorem~\ref{thm:ur-cartan-reln}, which establishes an ``Ur-Cartan" relation which we will eventually use to establish the all-important Cartan relation for the motivic dual Steenrod algebra in the sequel. This last result requires knowing that $D_n^{\mot}$ is op-lax monoidal which is easily seen for the equivariant model but poses an intimidating coherence problem for $D_n^{\mot}$ itself.

\subsection{Organization}
In \S\ref{sec:motivic-colim-first-props} we first define the motivic colimit functors, for diagrams valued in general presheaves of categories.
Then we show that they are in an appropriate sense left adjoint to ``constant diagram'' functors (Proposition \ref{prop:M-right-adjoint}), just like ordinary colimits.
As applications of this reinterpretation we study the compatibility of motivic colimits with localizing the category of indexing spaces $\PSh(\Sm_S)$ and we give an explicit formula for motivic colimits valued in presheaves.

In \S\ref{sec:motivic-colimits-and-exchange-trafos} we show that motivic colimits are compatible with the basic operations $f^*$, $f_\sharp$, $f_\otimes$ and also with changing the category.

In \S\ref{sec:motivic-kan-extensions} we develop a fragment of a theory of motivic Kan extensions.
Our main application is the construction of certain transfer maps between motivic colimits, provided that $\scr C$ satisfies ambidexterity for finite étale morphisms (Corollary \ref{cor:mot-colimits-transfers}).

In \S\ref{sec:motivic-extended-powers} we define the motivic extended powers as motivic colimits, and study their properties from this point of view.

In \S\ref{sec:equivariant-model} we show that motivic extended powers can be expressed in terms of genuine homotopy orbits, and use this to deduce some further properties.

In Appendix \ref{app:small-and-large-presheaves} we collect technical preliminaries about $\infty$-categories of presheaves on locally small categories.
These are surely well-known, but we could not locate references.

In Appendix \ref{subapp:kan-extension} we show that certain stacks, when viewed as presheaves on all affine schemes, are left Kan extended from smooth affine schemes.
These are some minor additions to \cite[Appendix A]{EHKSY3}, which we use to prove that motivic extended powers are stable under base change.

Finally in Appendix \ref{app:equiv-mot} we recall the construction and some basic properties of motivic equivariant homotopy theory.

\subsection{Notation and conventions}
We often have to deal with size issues, arising from the need to consider presheaves on categories which are not small.
Our conventions and some results for dealing with this issue are set out in Appendix \ref{sec:small-large-presheaves}.
In summary, by a large $\infty$-category one without restriction on the size of either the space of objects or spaces of morphisms.
By an $\infty$-category without further qualifications we mean a locally small one.
We write $\Spc$ for the $\infty$-category of small spaces and $\widehat{\Spc}$ for the $\infty$-category of large spaces, $\Cat_\infty$ for the $\infty$-category of small $\infty$-categories and $\widehat{\Cat}_\infty$ for the $\infty$-category of locally small $\infty$-categories.
Given a locally small $\infty$-category $\scr C$, we denote by $\PSh(\scr C)$ the full subcategory of $\Fun(\scr C^\op, \Spc)$ spanned by small colimits of representable functors. We write $\widehat{\PSh}(\scr C) 
= \Fun(\scr C^\op, \widehat{\Spc})$ for the category of large space valued presheaves. 

By ``cocontinuous'' or ``cocomplete'' without qualification we refer to properties with respect to small colimits.

We only use categorical language invariant under equivalences; in particular by ``small'' we mean ``essentially small'', and so on.

\subsection{Acknowledgements} Parts of this paper, and the power operations project, was written while the second author was a postdoc at the Center for Symmetry and Deformation at the University of Copenhagen, supported by the Danish National Research Foundation through the Centre for Symmetry and Deformation (DNRF92), which also funded a visit by the first author. We would also like to thank Tomer Schlank and Markus Spitzweck for useful discussions.

\section{Motivic colimits: definition and first properties} \label{sec:motivic-colim-first-props}
In this section, we generalize the \emph{``motivic Thom spectrum functor''} of \cite[Definition 16.1]{norms} by introducing the notion of \emph{motivic colimits}.
We place ourselves in the following context: we have a functor 
\begin{equation} \label{eq:c}
\scr C: \Sm_S^\op \to \widehat{\Cat}_\infty
\end{equation} subject to the following assumptions:
\begin{itemize}
\item If $p_X: X \to * \in \Sm_S$ is the unique map, then $\scr C(p_X): \scr C(S) \to \scr C(X)$ has a left adjoint $p_{X\sharp}$.
\item The category $\scr C(S)$ is cocomplete.
\end{itemize}
For $f: X \to Y \in \Sm_S$ we put $f^* := C(f): \scr C(Y) \to \scr C(X)$, when no confusion can arise.

\begin{remark} The situation above is the obvious $\infty$-categorical analog of an $\Sm$-fibered category in the sense of \cite[Section 1.1]{cisinski-deglise}; see \cite{ayoub} for another formulation. We note that this notion is also discussed in the language of this paper in the thesis of Khan \cite{adeelthesis}. Examples are aplenty: the unstable motivic homotopy $\infty$-category, its $S^1$- or $\P^1$-stable variants as well as various categories of modules over motivic ring spectra, and Voevodsky's version of the category of motives based on finite correspondences. We refer the reader to \cite{cisinski-deglise} for an encyclopedic reference.
\end{remark}

\subsection{Construction of the motivic colimit functors}
We let $(\Sm_S)_{\sslash \scr C}$ denote the source of the cartesian fibration $(\Sm_S)_{\sslash \scr C} \rightarrow \Sm_S$ classified by $\scr C$, i.e., $(\Sm_S)_{\sslash \scr C}$ is the category of elements of the functor $\scr C$. The $\infty$-category $(\Sm_S)_{\sslash \scr C}$ can be informally described in the following way:
\begin{itemize}
\item The objects are pairs $(X, E_X)$ with $X \in \Sm_S$ and $E_X \in \scr C(S)$.
\item The $1$-morphisms $(X, E_X) \rightarrow (Y, E_Y)$ consists of a morphism of smooth $S$-schemes $f: X \rightarrow Y$ and a morphism $\phi: E_X \rightarrow \scr C(f)(E_Y)$ in $\scr C(Y)$.
\end{itemize}
To begin our discussions of motivic colimits, we construct a functor \[M_0: (\Sm_S)_{\sslash \scr C} \to \scr C(S)\] which is informally described as follows:  it sends an object $(X, E) \in (\Sm_S)_{\sslash \scr C}$ to $p_{X\sharp} E \in \scr C(S)$. 

\begin{construction} \label{m0-construct}
By the Grothendieck construction, the fiber of $(\Sm_S)_{\sslash \scr C} \to \Sm_S$ over $S$ identifies with $\scr C(S)$, yielding an inclusion \[ M_0^R: \scr C(S) \to (\Sm_S)_{\sslash \scr C}. \]
By the assumptions on $\scr C$, this functor admits a left adjoint $M_0$ with the claimed description.
\end{construction}
Now, since $\scr C(S)$ is cocomplete, the functor $M_0$ extends to a cocontinuous functor \[M: \PSh((\Sm_S)_{\sslash \scr C}) \to \scr C(S);\] see Definition \ref{def:P-large-convention} and Lemma \ref{lemm:P-large}(2). The category $\PSh((\Sm_S)_{\sslash \scr C})$ is a bit unwieldy, so we will usually restrict $M$ to a smaller categories. By Lemma \ref{lemm:presheaf-slice-cat}, the slice category $(\Sm_S)_{/\scr C} \to \Sm_S$ is a cartesian fibration classifying the functor 
\[
\Sm_S^\op \to \widehat{\Cat}_\infty \qquad X \mapsto \scr C(X)^\wequi.
\] Composing with the co-unit of the adjunction $\widehat{\Spc} \adj \widehat{\Cat}_\infty$, we obtain a natural transformation $\scr C^\wequi \Rightarrow \scr C$, whence a canonical functor $(\Sm_S)_{/\scr C^\wequi} \to (\Sm_S)_{\sslash \scr C}$. Passing to presheaves, this induces $\PSh((\Sm_S)_{/\scr C^{\wequi}}) \to \PSh((\Sm_S)_{\sslash \scr C})$. The essential image of this functor is denoted by $\PSh(\Sm_S)_{\sslash \scr C}$. As the notation suggests, this can be identified with a right-lax slice category \cite[Appendix 5.1]{GRderalg}. Also note that $\PSh((\Sm_S)_{/\scr C}) \wequi \PSh(\Sm_S)_{/\scr C}$, by Lemma \ref{lemm:psh-slice-comp}.


\begin{definition}
We call the composites
\begin{align*}
\PSh((\Sm_S)_{\sslash \scr C}) &\xrightarrow{M} \scr C(S) \\
\PSh(\Sm_S)_{\sslash \scr C} \hookrightarrow \PSh((\Sm_S)_{\sslash \scr C}) &\xrightarrow{M} \scr C(S) \\
\PSh(\Sm_S)_{/\scr C^\wequi} \wequi \PSh((\Sm_S)_{/ \scr C^\wequi}) \to \PSh((\Sm_S)_{\sslash \scr C}) &\xrightarrow{M} \scr C(S),
\end{align*}
the \emph{motivic colimit functors} for $\scr C: \Sm_S^\op \to \widehat{\Cat}_\infty$. Depending on context, they are all denoted $M$, $M_S$ or $M^\scr{C}$.
\end{definition}

\begin{example}[Constant presheaves] \label{ex:constant-presheaves}
We have the terminal geometric morphism, whose inverse image is the constant functor and direct image is global sections \[ c: \Spc \rightleftarrows \PSh(\Sm_S): \Gamma. \]
Hence, if $\scr X \in \Spc$, a map $\alpha: c\scr X \rightarrow \scr C^\wequi$  in  $\widehat\PSh(\Sm_S)$ is equivalent to a functor 
\[
\bar\alpha: \scr X \rightarrow \scr C(S),
\]
i.e., a $\scr X$-shaped diagram in the category $\scr C(S)$. Since $M$ is cocontinuous and $M(* \xrightarrow{E} \scr C) = E$ one finds \[ M(\alpha) \wequi \colim_{\scr X} \bar\alpha. \] Hence in this case the motivic colimit functor simply computes the colimit of the functor $\bar\alpha$.
\end{example}

\begin{remark} \label{rmk:compute-motivic-colim-explicit}
Let $(\scr X \xrightarrow{\alpha} \scr C) \in \PSh(\Sm_S)_{/\scr C}$.
Then we have \[ (\scr X \xrightarrow{\alpha} \scr C) = \colim_{(x, X) \in (\Sm_S)_{ \sslash \scr X}} (X \xrightarrow{\alpha(x)} \scr C) \] and consequently \[ M(\scr X \xrightarrow{\alpha} \scr C) \wequi \colim_{(x, X) \in (\Sm_S)_{ \sslash \scr X}} (X \to S)_\sharp \alpha(x). \]
This generalizes the formula from Example \ref{ex:constant-presheaves}.
\end{remark}

\begin{remark}
Motivic colimits are special cases of \emph{relative colimits} \cite[\S4.3.1]{HTT}; see \cite[Remark 16.4]{norms}.
\end{remark}

\subsection{The right adjoint} In this section, we describe the motivic colimit functor as a partial left adjoint. Recall that if $F: \scr C \rightarrow \scr D$ is a functor, then its \emph{partial left adjoint} is defined on the full subcategory of $\scr D$ whose objects are $d$ such that the functor 
$\Map(d, F(-)): \scr C \rightarrow \Spc$ is representable \cite[Lemma 5.2.4.1]{HTT}.

\begin{example}
Given an $\infty$-category $\scr C$, the colimit functor $\colim: \Spc_{/\scr C} \rightarrow \scr C$ is a partial left adjoint to the functor which sends $c \in \scr C$ to the constant diagram at $c$.
\end{example}

\begin{proposition} \label{prop:M-right-adjoint}
The motivic colimit functor $M: \PSh(\Sm_S)_{/\scr C} \to \scr C(S)$ is a partial left adjoint of the functor
\[ \scr C(S) \to \widehat{\PSh}(\Sm_S)_{/\scr C}, \qquad E \mapsto \scr F_E, \]
where $\scr F_E \in \widehat{\PSh}(\Sm_S)_{/\scr C}$ is the presheaf with $\scr F_E(X) = \left( \scr C(X)_{/p_X^* E} \right)^\wequi$ and the source map to $\scr C(X)^\wequi$.
\end{proposition}
\begin{proof}
Recall that $M$ is obtained as the composite \[ \PSh(\Sm_S)_{/\scr C} \wequi \PSh((\Sm_S)_{/\scr C}) \to \PSh((\Sm_S)_{\sslash \scr C}) \xrightarrow{M'} \scr C(S), \] where $M'$ is the cocontinuous extension of a certain functor $M_0: (\Sm_S)_{\sslash \scr C} \to \scr C(S)$. Consequently $M'$ is a partial left adjoint to $M_0^*$, by Lemma \ref{lemm:big-P-partial-adj}. The functor $\widehat{\PSh}((\Sm_S)_{/\scr C}) \to \widehat{\PSh}((\Sm_S)_{\sslash \scr C})$ itself is left adjoint to pullback. It follows that $M$ is indeed a partial left adjoint, namely to pullback along $M_1: (\Sm_S)_{/\scr C} \to (\Sm_S)_{\sslash \scr C} \xrightarrow{M_0} \scr C(S)$.
This is just a more formal description of the functor $E \mapsto \scr F_E$: given $(X, F) \in (\Sm_S)_{/\scr C}$ (in other words $X \in \Sm_S$ and $F \in \scr C(X)$), we have
\begin{align*}
 M_1^*(E)((X, F)) &\wequi \Map_{\scr C(S)}(p_{X\sharp} F, E) \\
                  &\wequi \Map_{\scr C(X)}(F, p_X^* E) \\
                  &\wequi (\scr C(X)_{/p_X^* E})^\wequi \times_{\scr C(X)^\wequi} \{F\} \\
                  &\wequi \Map_{\widehat{\PSh}(\Sm_S)_{/\scr C^\wequi}}((X, F), \scr F_E).
\end{align*}
\end{proof}

\subsection{Localizing the source}
Before stating our next result, we need some preparation.
Note that we have a continuous extension of $\scr C$ which takes the form \[ \widehat{\scr C}: \PSh(\Sm_S)^\op \to \widehat{\Cat}_\infty.  \]
We also have the source functor $U: \PSh(\Sm_S)_{/\scr C} \to \PSh(\Sm_S)$.
\begin{example} \label{ex:UFE-expression}
For $E \in \scr C(S)$ we have $U\scr F_E \in \widehat\PSh(\Sm_S)$ and the functor $\Map(\ph, U\scr F_E): \PSh(\Sm_S) \to \widehat\Spc$, given by the continuous extension of \[ \Sm_S \ni X \mapsto (\scr C(X)_{/p_X^*E})^\wequi. \]
Since passage to slice categories and maximal subgroupoids preserves limits\NB{$\scr D_{/x} = \Fun(\Delta^1, \scr D) \times_{\scr D} \{x\}$}, this continuous extension is given by \[ \Map(\scr X, U\scr F_E) \wequi (\widehat{\scr C}(\scr X)_{/p_\scr{X}^*E})^\wequi. \]
\end{example}

\begin{proposition} \label{corr:M-local-equiv}
Let $W$ be a set of morphisms in $\PSh(\Sm_S)$ and assume that (the continuous extension of) $\scr C$ is $W$-local.
\begin{enumerate}
\item The localization of $\PSh(\Sm_S)_{/\scr C}$ at (the strong saturation of) $U^{-1}(W)$ exists and is equivalent to $\PSh(\Sm_S)[W^{-1}]_{/\scr C}$.
\item The functor $M: \PSh(\Sm_S)_{/\scr C} \to \scr C(S)$ preserves $U^{-1}(W)$-local weak equivalences and hence factors through $\PSh(\Sm_S)_{/\scr C} \to \PSh(\Sm_S)[W^{-1}]_{/\scr C}$.
\end{enumerate}
\end{proposition}
\begin{proof}
(1)
The localization exists and coincides with the subcategory of $W^{-1}(U)$-local objects, by presentability reasons and the adjoint functor theorem.\todo{really?}
We shall show that $(G \to \scr C) \in \PSh(\Sm_S)_{/\scr C}$ is $W^{-1}(U)$-local if and only if $G$ is $W$-local.
Thus let $F_1 \to F_2 \in \PSh(\Sm_S)$ be a $W$-equivalence, $F_2 \to \scr C$ any map and consider the induced map $(F_1 \to \scr C) \to (F_2 \to \scr C) \in \PSh(\Sm_S)_{/\scr C}$.
Note that $\Map((F_i \to \scr C), (G \to \scr C))$ is the fiber of $\Map(F_i, G) \to \Map(F_i, \scr C)$.
Since $\scr C$ is $W$-local, the base is independent of $i$, and hence \[ \Map((F_1 \to \scr C), (G \to \scr C)) \wequi \Map((F_2 \to \scr C), (G \to \scr C)) \] for all choices of map $F_2 \to \scr C$ (i.e. base points) if and only if $\Map(F_1, G) \wequi \Map(F_2, G)$.
This was to be shown.

(2) Both statements are equivalent to $M^*(\scr C(S)) \subset \widehat{\PSh}(\Sm_S)[W^{-1}]_{/\scr C} \subset \widehat{\PSh}(\Sm_S)_{/\scr C}$. In other words, by Proposition \ref{prop:M-right-adjoint}, we need to check that for $E \in \scr C(S)$ the big presheaf $\scr F_E \in \widehat{\PSh}(\Sm_S)$ is $W$-local. That is, for $\alpha: \scr X \to \scr Y \in W$ we know that $\alpha^*: \scr C(\scr Y) \to \scr C(\scr X)$ is an equivalence, and by Example \ref{ex:UFE-expression} we need to prove that $\left(\scr C(\scr Y)_{/p_{\scr Y}^* E}\right)^\wequi \to \left(\scr C(\scr X)_{/p_{\scr X}^* E}\right)^\wequi$ is an equivalence. Since $p_{\scr X}^* E \wequi \alpha^* p_{\scr Y}^* E$, this is clear.
\end{proof}

\begin{example} \label{ex:M-P-Sigma}
Suppose that $\scr C: \Sm_S^\op \to \widehat{\Cat}_\infty$ preserves finite products. Then $M: \PSh(\Sm_S)_{/\scr C} \to \scr C(S)$ factors through $\PSh(\Sm_S)_{/\scr C} \to \PSh_\Sigma(\Sm_S)_{/\scr C}$. Indeed $\PSh_\Sigma(\Sm_S)$ is obtained from $\PSh(\Sm_S)$ by localizing at the set of map $\emptyset \to r_\emptyset$ and $r_X \coprod r_Y \to r_{X \coprod Y}$, where $r: \Sm_S \to \PSh(\Sm_S)$ denotes the yoneda functor, and these correspond in the continuous extension of $\scr C$ precisely to $\scr C(\emptyset) \to *$ and $\scr C(X \coprod Y) \to \scr C(X) \times \scr C(Y)$.
\end{example}

\begin{example} \label{ex:M-tau-local}
If $\tau$ is a topology on $\Sm_S$ and $\scr C$ is a $\tau$-sheaf, then $M$ factors through $\PSh(\Sm_S)_{/\scr C} \to \Shv_\tau(\Sm_S)_{/\scr C}$. A similar remark applies for motivic/$\A^1$-localization. However, in our examples of interest, $\scr C$ is not $\A^1$-invariant (for example, $\SH(X \times \A^1)$ is not equivalent to $\SH(X)$; \cite[Remark 16.10]{norms}) and thus the motivic colimit functor is \emph{not} $\A^1$-invariant in the diagram variable. 
\end{example}

\subsection{The case $\scr C(X) = \PSh(\Sm_X)$.}
\label{subsec:M-PSh}

We now investigate motivic colimits for the functor $\scr C(\ph) = \PSh(\Sm_{(\ph)})$, i.e., the functor 
\[
\scr C: \Sm^{\op}_S \rightarrow \widehat{\Cat}_\infty; X \mapsto \PSh(\Sm_X).
\]
In this case, we can give a formula for the motivic colimit functor; we do this in Corollary~\ref{corr:M-compute}.

Consider the composite functor \[ \gamma: \Sm_S \xrightarrow{c} (\Sm_S)_{\sslash \Sm_{(\ph)}} \hookrightarrow (\Sm_S)_{\sslash \PSh(\Sm_{(\ph)})}. \]
Here $c$ is the functor $X \mapsto (X \in \Sm_S, X \in \Sm_X)$, which we can write down by hand since $(\Sm_S)_{\sslash \Sm_{(\ph)}}$ is a $(2,1)$-category.\NB{or see explanation in commented out line below this}

\begin{proposition} \label{prop:M-corep}
The motivic colimit functor $M: \PSh((\Sm_S)_{\sslash \PSh(\Sm_{(\ph)})}) \to \PSh(\Sm_S)$ is right adjoint to the cocontinuous extension \[ \gamma: \PSh(\Sm_S) \to \PSh((\Sm_S)_{\sslash \PSh(\Sm_{(\ph)})}). \]
\end{proposition}
\begin{proof}
The composite $M\gamma$ is a colimit preserving endofunctor of $\PSh(\Sm_S)$ and hence determined by its restriction to $\Sm_S$.
Since $M\gamma X = M(X \xrightarrow{X} \PSh(\Sm_X)) = X$ we find that $M\gamma \wequi \id$; we shall show that this equivalence is a counit of adjunction.
In other words for $\scr X \in \PSh(\Sm_S)$ and $\scr G \in \PSh((\Sm_S)_{\sslash \PSh(\Sm_{?})})$ we need to show that $\Map(\gamma \scr X, \scr G) \wequi \Map(\scr X, M \scr G)$.
We may assume that $\scr X = X \in \Sm_S$; thus $X, \gamma X$ are completely compact\footnote{Let us that $\scr C$ is a large $\infty$-category (using the convention of Appendix~\ref{app:small-and-large-presheaves}) then an object $X \in \scr C$ is \emph{completely compact} if $\Map(X,-):\scr C \rightarrow \widehat{\Spc}$ commutes with all small colimits.} and we may assume that $\scr G = (U, E)$ with $U \in \Sm_S$ and $E \in \PSh(\Sm_Y)$.
There is a canonical map $p_{U\sharp} E \to p_{U\sharp} * = U$ inducing $P': \Map_{\PSh(\Sm_S)}(X, p_{U\sharp} E) \to \Map_{\PSh(\Sm_S)}(X, U) \wequi \Map_{\Sm_S}(X, U).$
From this we obtain the following commutative diagram
\begin{equation*}
\begin{tikzcd}
\Map_{(\Sm_S)_{\sslash \PSh(\Sm_{(\ph)})}}((X, *), (U, E)) \ar[r,"M"] \ar[d, "P"] & \Map_{\PSh(\Sm_S)}(X, p_{U\sharp} E) \ar[dl, "P'"] \\
\Map_{\Sm_S}(X, U),
\end{tikzcd}
\end{equation*}
where $P$ is the tautological functor $(\Sm_S)_{\sslash \PSh(\Sm)} \to \Sm_S$. In order to prove that the top map is an equivalence, it suffices to prove that it induces an equivalence on the fiber over any point $f \in \Map_{\Sm_S}(X, U)$. The fiber of $P$ over $f$ is $\Map_{\PSh(\Sm_X)}(*, f^* E) \wequi \Map_{\PSh(\Sm_S)_{/X}}(X, X \times_{f,U} p_{U\sharp} E)$. Here we have used that $\Sm_X \to (\Sm_S)_{/X}$ is fully faithful, and hence so is its cocontinuous extension. This is the same as the fiber of $P'$ over $f$, essentially by definition. This concludes the proof.
\end{proof}

\begin{corollary} \label{corr:M-compute}
Let $\scr X \in \PSh(\Sm_S)$ and $\alpha: \scr X \to \PSh(\Sm_{(\ph)})$, in other words $\alpha \in \PSh(\Sm_S)_{/\PSh(\Sm_{(\ph)})}$. Then for $U \in \Sm_S$ we have
\[ M(\alpha)(U) \wequi \colim_{\scr X(U)} \alpha_U(\ph)(U). \]
Here $\alpha_U(\ph): \scr X(U) \to \PSh(\Sm_U)$ is the functor obtained from $\alpha$ by taking sections over $U$.
\end{corollary}
\begin{proof}
By Proposition \ref{prop:M-corep} we have
\[ M(\alpha)(U) \wequi \Map_{\PSh(\Sm_S)_{\sslash \PSh(\Sm_{(\ph)})}}((U, *), \alpha). \]
The result now follows from the formula for computing mapping spaces in lax slice categories\NB{ref; i.e. $\Map((a\xrightarrow{\alpha} x), (b \xrightarrow{\beta} x)) \wequi \colim_{f \in \Map(a,b)} Nat(f\beta, \alpha)$}. Alternatively we can argue as follows. Since $(U, *)$ is completely compact and the colimit functor $\Spc_{/\Spc} \to \Spc$ is cocontinuous (being the cocontinuous extension of a functor $*_{/\Spc} \to \Spc$), both sides are stable under colimits in the $\alpha$ variable. Since $\PSh(\Sm_S)_{/\PSh(\Sm_{(\ph)})} \to \PSh(\Sm_S)_{\sslash \PSh(\Sm_{(\ph)})}$ preserves colimits and the left hand side is generated under colimits by objects of the form $(X, E \in \PSh(\Sm_X))$, the problem reduces to $\alpha$ of this form. In this case we have seen in the proof of Proposition \ref{prop:M-corep} that there is a cartesian fibration in spaces $P: \Map_{(\Sm_S)_{\sslash \PSh(\Sm)}}((U, *), (X, E)) \to \Map_{\Sm_S}(U, X)$. This classifies the functor $\Map_{\Sm_S}(U, X) \to \Spc, f \mapsto \Map_{\PSh(\Sm_X)}(*, f^* E) = \alpha_U(f)(U)$. The source of a cartesian fibration over an $\infty$-groupoid is a model for the colimit of the associated functor by (the dual of) \cite[Corollary 3.3.3.3]{HTT}, so the result follows.
\end{proof}

\begin{example}
Let $\scr X \in \PSh(\Sm_S)$. We define $\scr X_* \in \PSh(\Sm_S)_{/\PSh(\Sm_{(\ph)})}$ as the composite $\scr X \to * \to \PSh(\Sm_{(\ph)})^\wequi$, where the second map picks out the object $* \in \PSh(\Sm_S)$. By Corollary \ref{corr:M-compute} we find
\[ M(\scr X_*)(U) \wequi \colim_{\scr X(U)} * \wequi \scr X(U), \]
whence $M(\scr X_*) \wequi \scr X$. (This is also easy to see directly from the definition of $M$.)
\end{example}

\section{Motivic colimits and exchange transformations} \label{sec:motivic-colimits-and-exchange-trafos}

%

\subsection{Changing $\scr C$}
Suppose given $\eta: \scr C \to \scr D \in \Fun(\Sm_S^\op, \widehat{\Cat}_\infty)$. Under suitable assumptions there are induced motivic colimit functors which we denote by $M_\scr{C}: \PSh((\Sm_S)_{\sslash \scr C}) \to \scr C(S)$ and $M_\scr{D}: \PSh((\Sm_S)_{\sslash \scr D}) \to \scr D(S)$. Passing to the associated cartesian fibrations, $\eta$ induces $\bar\eta: (\Sm_S)_{\sslash \scr C} \Rightarrow (\Sm_S)_{\sslash \scr D}$. We also denote by $\bar\eta: \PSh((\Sm_S)_{\sslash \scr C}) \to \PSh((\Sm_S)_{\sslash \scr D})$ the cocontinuous extension.

\begin{remark}
Since $(\Sm_S)_{/\scr C} \wequi (\Sm_S)_{\sslash \scr C^\wequi}$, the commutative diagram
\begin{equation*}
\begin{CD}
\scr C^\wequi @>\eta^\wequi>> \scr D^\wequi \\
@VVV                               @VVV     \\
\scr C @>\eta>> \scr D \\
\end{CD}
\end{equation*}
in $\Fun(\Sm_S^\op, \widehat{\Cat}_\infty)$ induces a commutative diagram
\begin{equation*}
\begin{CD}
(\Sm_S)_{/\scr C} @>\bar\eta>> (\Sm_S)_{/\scr D} \\
@VVV                               @VVV     \\
(\Sm_S)_{\sslash \scr C} @>\bar\eta>> (\Sm_S)_{\sslash \scr D}.
\end{CD}
\end{equation*}
\end{remark}

\begin{construction} \label{cons:ex-trans} 
Suppose given a square of $\infty$-categories
\begin{equation*}
\begin{CD}
\scr C_1 @>F_1>> \scr C_2  \\
@AG_1AA            @AG_2AA \\
\scr C_3 @>F_2>> \scr C_4,
\end{CD}
\end{equation*}
commuting up a specified equivalence $\alpha$. Suppose further that $G_1$ has a left adjoint $M_1$ and $G_2$ has a left adjoint $M_2$. Then there is a canonical exchange transformation (also usually called the Beck-Chevalley transformation; see also \cite[Definition 4.7.4.13]{HA}) $M_2 F_1 \Rightarrow F_2 M_1$, namely \[ M_2F_1 \stackrel{u}{\Rightarrow} M_2F_1G_1M_1 \stackrel{\alpha}{\simeq} M_2G_2F_2M_1 \stackrel{c}{\Rightarrow} F_2M_1. \]
\end{construction}

\begin{proposition} \label{prop:changing-C} Let $\scr C, \scr D: \Sm_S^\op \to \widehat\Cat_{\infty}$ satisfy the assumptions of~\eqref{eq:c}. Suppose that we have a morphism $\eta: \scr C \rightarrow \scr D$ such that the functor $\eta_S: \scr C(S) \rightarrow \scr D(S)$ preserves all small colimits. Then there are canonical exchange transformations:
\begin{equation} \label{ex-pt-wise}
\Ex_{\sharp\eta}: p_{X\sharp}^\scr{D} \eta_X \Rightarrow \eta_S p_{X\sharp}^\scr{C}: \scr C(X) \rightarrow \scr D(S),
\end{equation}
and
\begin{equation} \label{ex-mot-colim}
\Ex_{M\eta}: M^{\scr D} \bar\eta \Rightarrow \eta_S M^\scr{C}: \PSh((\Sm_S)_{\sslash \scr C}) \to \scr D(S).
\end{equation}
Moreover, if each $\Ex_{\sharp\eta}$ is an equivalence, then so is $\Ex_{M\eta}$.
\end{proposition}
\begin{proof} The exchange transformation $\Ex_{\sharp\eta}$ of~\eqref{ex-pt-wise} comes directly from Construction \ref{cons:ex-trans}. We proceed to construct~\eqref{ex-mot-colim}.
Recall Construction \ref{m0-construct} of $M_0^R$.
By functoriality of the Grothendieck construction we have a commutative diagram
\begin{equation*}
\begin{CD}
\scr C(S) @>M_0^R>>   (\Sm_S)_{\sslash \scr C} \\
@V{\eta_S}VV      @V{\bar \eta}VV     \\
\scr D(S) @>M_0^R>>   (\Sm_S)_{\sslash \scr D} \\
\end{CD}
\end{equation*}
and hence Construction \ref{cons:ex-trans} supplies us with an exchange transformation
\[
\Ex_{0,M\eta}: M_0^{\scr D} \bar\eta \Rightarrow \eta_S M_0^\scr{C}: (\Sm_S)_{\sslash \scr C} \to \scr D(S).
\]
The transformation $\Ex_{0, M\eta}$ extends canonically to a transformation
\[
\Ex_{M\eta}: M^{\scr D} \bar\eta \Rightarrow \eta_S M^\scr{C}: \PSh((\Sm_S)_{\sslash \scr C}) \to \scr D(S)
\]
(using Lemma \ref{lemm:P-large}(2)); we use the fact that $\eta_S$ preserves small colimits to ensure that the extension of $(\eta_S M_0^{\scr{C}})$ to $ \PSh((\Sm_S)_{\sslash \scr C})$ agrees with the composite $\eta_S \circ M^{\scr{C}}$. The last statement follows since, by construction, if $(X, E_X) \in \PSh((\Sm_S)_{\sslash \scr C})$ is representable then $\Ex_{M\eta}(X, E_X) \simeq \Ex_{\sharp\eta}(E_X)$.

%
%
%
\end{proof}

\begin{example} \label{ex:changing-C}
We have a sequence of natural transformations in $\Fun(\Sm_S^\op, \widehat{\Cat}_\infty)$
\[ \PSh(\Sm_{(\ph)}) \Rightarrow \PSh_\Sigma(\Sm_{(\ph)}) \Rightarrow \Spc(\ph) \Rightarrow \Spc(\ph)_* \Rightarrow \SH(\ph) \Rightarrow \DM(\ph, \Z) \Rightarrow \DM(\ph, \Z/2), \]
where each transformation is the obvious one, and satisfies all the assumptions of Proposition \ref{prop:changing-C}. In other words, passage from any category in the list to one further to the right is compatible with motivic colimits (just like it is compatible with ordinary colimits).
\end{example}

\subsection{Changing $S$}
Suppose that $\scr C$ is defined on a larger category than $\Sm_S$, e.g. $\scr C: \Sch^\op \to \widehat{\Cat}_\infty$. For $S \in \Sch$ write $\scr C_S: \Sm_S \to \Sch \to \widehat{\Cat}_\infty$ for the restriction of $\scr C$ to $\Sm_S$. If $\scr C_S$ satisfies the hypotheses of~\eqref{eq:c}, there is an induced motivic colimit functor 
\[M_S: \PSh((\Sm_S)_{\slash \scr C_S}) \to \scr C(S).
\] 
for each $S \in \Sch$. In this section, we fix such as $\scr C$ study the variance of of $M_S$ as $S$ varies, i.e, if $f: T \to S \in \Sch$ then we compare $M_S$ and $M_T$. We first construct 
\[f^*: (\Sm_S)_{\sslash \scr C_S} \to (\Sm_T)_{\sslash \scr C_T}.
\] It can be informally described as $(X, E) \mapsto (X \times_S T, f^* E)$. 
\begin{construction} \label{f*-construct} We have a morphism of schemes $f: T \rightarrow S$.
The functor $f_\sharp: \Sch_T \to \Sch_S$ induces by functoriality of the Grothendieck construction (on the base, as in \cite[Corollary A.31]{ghn-lax}) a morphism $(\Sch_T)_{\sslash \scr C} \to (\Sch_S)_{\sslash \scr C}$, informally described as $(X \in \Sch_T, E \in \scr C(X)) \mapsto (X \in \Sch_S, E \in \scr C(X))$.
This admits a right adjoint $(\Sch_S)_{\sslash \scr C} \to (\Sch_T)_{\sslash \scr C}$ informally described as $(X, E) \mapsto (X \times_S T, f^* E)$.
The desired functor $f^*$ is obtained by restricting to the full subcategories $(\Sm_{(\ph)})_{\sslash \scr C_{(\ph)}} \subset (\Sch_{(\ph)})_{\sslash \scr C}$.
\end{construction}
\begin{proof}
In order to see that the right adjoint exists and determine its value, let $Y \in \Sch_S$ and $F \in \scr C(Y)$.
A morphism from $f_\sharp(X, E)$ to $(Y, F)$ consists of an $S$-morphism $p: X \to Y$ and a map $E \to p^*F$ in $\scr C(X)$.
This is the same data as a $T$-morphism $p': X \to Y \times_S T$ together with a map $E \to p'^* q^*F$, where $q$ is the projection $Y \times_S T \to Y$.
Consequently $f^*(Y,F)$ is represented by $(Y \times_S T, q^* F)$, as desired.
\end{proof}

%
\begin{variant}
In precisely the same way, we construct $f^*: (\Sm_S)_{/\scr C_S} \to (\Sm_T)_{/\scr C_T}$, which fits into a commutative diagram
\begin{equation*}
\begin{CD}
(\Sm_S)_{/\scr C_S} @>f^*>> (\Sm_T)_{/\scr C_T} \\
@VVV                          @VVV              \\
(\Sm_S)_{\sslash \scr C_S} @>f^*>> (\Sm_T)_{\sslash \scr C_T}.
\end{CD}
\end{equation*}
\end{variant}

\begin{proposition} \label{prop:changing-S} \NB{assumptions not quite minimal}
Suppose that we have a functor $\scr C: \Sch_S^\op \to \widehat{\Cat}_\infty$ such that
\begin{enumerate}
\item For each $T \in \Sch_S$ the category $\scr C(T)$ admits small colimits.
\item For each $f: T' \to T \in \Sch_S$ the functor $f^*:\scr C(T) \to \scr C(T')$ preserves small colimits.
\item If $f$ is smooth then $f^*$ has a left adjoint $f_\sharp$.
\end{enumerate}

Let $f: T' \to T \in \Sch_S$. Consider the diagram
\begin{equation*}
\begin{CD}
\PSh((\Sm_{T})_{\sslash \scr C_T}) @>M_T>> \scr C(T) \\
@Vf^*VV                                  @Vf^*VV   \\
\PSh((\Sm_{T'})_{\sslash \scr C_{T'}}) @>M_T'>> \scr C(T'),
\end{CD}
\end{equation*}
where the left vertical morphism $f^*$ is the cocontinuous extension of the functor 
\[f^*: (\Sm_S)_{\sslash \scr C_T} \to (\Sm_T)_{\sslash \scr C_{T'}}
\] defined in Construction~\ref{f*-construct}. Then,
\begin{enumerate}
\item there is a canonical exchange transformation 
\[
\Ex_M^*: M_{T'} f^* \Rightarrow f^* M_T : \PSh((\Sm_S)_{\sslash \scr C_S}) \to \scr C(T).
\]
\item If $\scr C$ satisfies \emph{smooth base change} over $T$, i.e. for any smooth morphism $g: X \to T$ the exchange transformation $ex^*_\sharp: g'_\sharp f'^* \to f^* g_\sharp: \scr C(X) \to \scr C(T')$ is an equivalence, then also $\Ex_M^*$ is an equivalence.
\end{enumerate}
\end{proposition}
\begin{proof}
(1) By constriction the functor $f^*: (\Sm_S)_{\sslash \scr C_T} \to (\Sm_T)_{\sslash \scr C_{T'}}$ maps the fiber above $T$ to the fiber above $T'$ and induces the functor $f^*: \scr C(T) \to \scr C(T')$ on these fibers.
In other words, the following diagram commutes
\begin{equation*}
\begin{CD}
\scr C(T) @>M_{0T}^R>> (\Sm_T)_{\sslash \scr C_T} \\
@Vf^*VV                @Vf^*VV \\
\scr C(T') @>M_{0T'}^R>> (\Sm_{T'})_{\sslash \scr C_{T'}} \\
\end{CD}
\end{equation*}
Construction \ref{cons:ex-trans} supplies us with an exchange transformation  \[\Ex_{0M}^*: M_{0T'} f^* \Rightarrow f^* M_{0T},\] which extends canonically 
to a transformation
\[
\Ex^*_M: M_{T'}f^* \Rightarrow f^*  M_{T}: \PSh((\Sm_T)_{\sslash \scr C_T}) \to \scr C(T')
\]
(using Lemma \ref{lemm:P-large}(2)); we use the fact that $f^*: \scr C(T) \rightarrow \scr C(T')$ preserves small colimits to ensure that the extension of $f^*  M_{T}$ to $ \PSh((\Sm_T)_{\sslash \scr C})$ agrees with the composite $f^* \circ M_T$. 

(2) It is then enough to show that $\Ex^*_M$ is an equivalence on all objects $(X, E) \in (\Sm_T)_{\sslash \scr C_T}$ (i.e. $X \in \Sm_T$ and $E \in \scr C(X)$). It follows from the construction that this is precisely $\Ex^*_\sharp(E)$.
\end{proof}

\begin{remark} \label{ex:changing-S}
Each of the functors $\Sm_S^\op \to \widehat{\Cat}_\infty$ from Example \ref{ex:changing-C} has an evident extension to $\Sch^\op$ satisfying all the assumptions of Proposition \ref{prop:changing-S}. In other words, motivic colimits in any of these categories commute with arbitrary base change (just like ordinary colimits).
\end{remark}

\subsection{Interaction with $f_\sharp$} \label{subsec:interaction-M-f-sharp}

Let $f: T \to S \in \Sm_S$. We have a cartesian square
\begin{equation*}
\begin{CD}
(\Sm_T)_{\sslash \scr C_T} @>>> (\Sm_S)_{\sslash \scr C} \\
@VVV                            @VVV                     \\
\Sm_T                      @>{f_\sharp}>> \Sm_S.
\end{CD}
\end{equation*}
We denote the top horizontal arrow by $f_\sharp: (\Sm_T)_{\sslash \scr C_T} \to (\Sm_S)_{\sslash \scr C}$.

\begin{proposition} \label{prop:interaction-M-f-sharp}
Let $f: T \to S \in \Sm_S$.
\begin{enumerate}
\item The cocontinuous extension $f_\sharp: \PSh((\Sm_T)_{\sslash \scr C_T}) \to \PSh((\Sm_S)_{\sslash \scr C})$ induces a commutative diagram
\begin{equation*}
\begin{CD}
\PSh(\Sm_T)_{/ \scr C_T} @>>> \PSh(\Sm_T)_{\sslash \scr C_T} @>>> \PSh((\Sm_T)_{\sslash \scr C_T}) \\
@V{f_\sharp}VV @V{f_\sharp}VV @V{f_\sharp}VV \\
\PSh(\Sm_S)_{/ \scr C} @>>> \PSh(\Sm_S)_{\sslash \scr C} @>>> \PSh((\Sm_S)_{\sslash \scr C}).
\end{CD}
\end{equation*}
\item The functor $f_\sharp: \PSh((\Sm_T)_{\sslash \scr C_T}) \to \PSh((\Sm_S)_{\sslash \scr C})$ has a right adjoint $f^*$ which satisfies $f^*((\Sm_S)_{\sslash \scr C}) \subset (\Sm_T)_{\sslash \scr C_T}$.
  In fact for $X \in \Sm_S$ and $a \in \scr C(X)$ we have $f^*(X, a) = (X \times_S T, p^* a)$ where $p: X \times_S T \to X$ is the canonical projection.
  Moreover $f^*: \PSh((\Sm_S)_{\sslash \scr C}) \to \PSh((\Sm_T)_{\sslash \scr C_T})$ restricts to $f^*: \PSh(\Sm_S)_{\sslash \scr C} \to \PSh(\Sm_T)_{\sslash \scr C_T}$ and is compatible with $f^*: \PSh(\Sm_S)_{/ \scr C} \to \PSh(\Sm_T)_{/ \scr C_T}$.
\item Suppose that $M_S$ and $M_T$ exist. Then there is a canonical isomorphism $ex_{\sharp M}: f_\sharp M_T \wequi M_S f_\sharp \in \Fun(\PSh((\Sm_T)_{\sslash \scr C_T}), \scr C(S))$.
\end{enumerate}
\end{proposition}
\begin{proof}
(1) Since $\PSh(\Sm_T)_{\sslash \scr C_T} \to \PSh((\Sm_T)_{\sslash \scr C_T})$ is fully faithful by definition, and similarly for $S$, it suffices to construct the outer commutative rectangle. Since all functors are cocontinuous, it suffices to construct
\begin{equation*}
\begin{CD}
(\Sm_T)_{/\scr C_T} @>>> (\Sm_T)_{\sslash \scr C_T} \\
@VVV                       @VVV                     \\
(\Sm_S)_{/\scr C} @>>> (\Sm_S)_{\sslash \scr C}.
\end{CD}
\end{equation*}
This is just obtained by pulling back the map $\scr C^\wequi \to \scr C$ along $f_\sharp: \Sm_T \to \Sm_S$, and taking the associated cartesian fibrations.

(2) This is a special case of Construction \ref{f*-construct}.

(3) In order to construct $ex_{\sharp M}: f_\sharp M_T \Rightarrow M_S f_\sharp$, we may by adjunction construct $M_T \Rightarrow f^* M_S f_\sharp$. This we take to be the composite $M_T \xRightarrow{M_T \eta} M_T f^* f_\sharp \xRightarrow{ex_M^* f_\sharp} f^* M_S f_\sharp$, where $\eta$ is the unit of the adjunction $f_\sharp \dashv f^*$ and $ex_M^*$ is from Proposition \ref{prop:changing-S}. In order to prove that $ex_{\sharp M}$ is an isomorphism, since all functors involved are cocontinuous, we may restrict to $(\Sm_T)_{/\scr C}$. Thus let $p: X \to T \in \Sm_T$ and $E \in \scr C(X)$. Then $f_\sharp M_T (X, E) \wequi f_\sharp p_\sharp E \wequi (fp)_\sharp E \wequi M_S f_\sharp (X, E)$. This concludes the proof.
\end{proof}

\begin{remark}
The functor $f_\sharp: \PSh(\Sm_T)_{/\scr C_T} \to \PSh(\Sm_S)_{/\scr C}$ can be informally described as $(\scr X \to \scr C_T) \mapsto (f_\sharp \scr X \to f_\sharp \scr C_T \xrightarrow{\alpha} \scr C)$, where $\alpha: f_\sharp \scr C_T \to \scr C$ corresponds by adjunction to the isomorphism $\scr C_T \to f^* \scr C$.
\end{remark}

\subsection{Interaction with norms}
\label{subsec:interaction-colimits-norms}

In this subsection, we will study the situation where $\scr C$ is provided with an additional covariant functoriality for finite étale morphisms. In other words, we assume given a commutative diagram
\begin{equation*}
\begin{tikzcd}
\Sm_S^\op \ar[r] \ar[d, "\scr C"] & \Span(\Sm_S, \all, \fet) \ar[dl, "\scr C^\otimes"] \\
\widehat{\Cat}_\infty.
\end{tikzcd}
\end{equation*}
In this generality, we denote the extra functoriality by $p_\otimes$. This could be a norm map as constructed in \cite{norms}, or it could be given by $p_\sharp$ or $p_*$, for example.

We denote by $\NAlg(\scr C^\otimes)$ the $\infty$-category of sections of the cocartesian fibration classified by $\scr C^\otimes$, which are cocartesian over $\Sm_S^{\op} \hookrightarrow \Span(\Sm_S, \all, \fet).$ c.f. \cite[\S7]{norms}.

We now wish to repeat the discussion of \S\ref{subsec:interaction-M-f-sharp}, with $p_\otimes$ in place of $f_\sharp$. Thus let $p: T \to S \in \FEt_S$. We have a functor $p_*: \PSh(\Sm_T)_{\sslash \scr C_T} \to \PSh(\Sm_S)_{\sslash p_* \scr C_T}$. We also have a map $\alpha: p_* \scr C_T \to \scr C$, given on sections by $\alpha_X = p_{X\otimes}: p_* \scr C_T(X) \wequi \scr C(p^* X) \to \scr C(X)$. Postcomposing $p_*$ with $\alpha$ we obtain a functor $p_\otimes: \PSh(\Sm_T)_{\sslash \scr C_T} \to \PSh(\Sm_S)_{\sslash \scr C}$.

\begin{example}
Let $X \in \Sm_T$ such that the Weil restriction $R_p X \in \Sm_S$ exists (e.g. $X \in \SmQP_T$). Let $E \in \scr C(X)$. Then $p_\otimes (X, E) \wequi (R_p X, r_\otimes e^* E)$, where $e: p^* R_p X \to X$ is the counit of adjunction, and $r: p^* R_p X \to R_p X$ is the projection.
\end{example}

\begin{proposition} \label{prop:M-p-otimes}
Let $p: T \to S$ be finite étale and assume that $M_S, M_T$ exist. Assume that each $q_\otimes: \scr C(X) \to \scr C(Y)$ (for $q: X \to Y \in \Sm_S$ finite étale) preserves sifted colimits. Let $\scr P \subset \PSh(\Sm_T)_{\sslash \scr C_T}$ be the essential image of $\PSh_\Sigma(\SmQP_T)_{/\scr C_T}$. 
\begin{enumerate}
\item There is a canonical exchange transformation $ex_{M\otimes}: M_S p_\otimes \Rightarrow p_\otimes M_T \in \Fun(\scr P, \scr C(S))$.
\item Suppose that $\scr C$ satisfies the distributivity law of \cite[Proposition 5.12]{norms}. Then $ex_{M\otimes}$ is an isomorphism.
\end{enumerate}
\end{proposition}
\begin{proof}
The category $\scr P$ is equivalent to a full subcategory of $\PSh_\Sigma((\SmQP_T)_{\sslash \scr C_T})$. We have a functor $\tilde{p}_\otimes: (\SmQP_T)_{\sslash \scr C_T} \to (\SmQP_S)_{\sslash \scr C}$ given by $(X, E) \mapsto (R_p X, r_\otimes e^* E)$. The sifted-cocontinuous extension of $\tilde{p}_\otimes$ is compatible with the functor $p_\otimes$ we have constructed before. These observations together imply that we may replace $\scr P$ by $(\SmQP_T)_{\sslash \scr C_T}$. Then for (1) we need to construct $(R_p X \to S)_\sharp r_\otimes e^* E \to p_\otimes (X \to T)_\sharp E$, naturally in $(X, E) \in (\SmQP_T)_{\sslash \scr C_T}$. For this we can take the transformation $Dis_{\sharp\otimes}$ of \cite[Proposition 5.12]{norms}. For (2), we need to show that the transformation is an equivalence, which is the assumption.
\end{proof}

\begin{remark} By the main result of \cite{e-rune}, the assumption that $\scr C$ satisfies distributivity is equivalent to promoting the functor $\scr C$ from a functor out of a category of spans to the category of \emph{bispans}; see \cite[Section 3.5]{e-rune} for a discussion in the context of motivic homotopy theory.
\end{remark}

\begin{remark}
We were being slightly less careful in this subsection than in others, because actually the following stronger statement is true (under slightly stronger assumptions): the assignment 
$X \mapsto \PSh_\Sigma(\SmQP_X)_{\sslash \scr C_T}$ can be made 
functorial in 
$\Span(\Sm_S, \all, \fet)$, and then 
$M$ becomes a (strong) natural transformation \[\PSh_\Sigma(\SmQP_{(\ph)})_{\sslash \scr C_{(\ph)}} \Rightarrow \scr C^\otimes \in \Fun(\Span(\Sm_S, \all, \fet), \widehat{\Cat}_\infty).\] This is proved using the same arguments as in \cite[first half of \S16.3]{norms}.
\end{remark}

\section{Motivic left Kan extensions and $\FEt$-semiadditivity} \label{sec:motivic-kan-extensions} \NB{This is an alternative to the $(\infty,2)$-categories mess...}

In this section, we give the first working demonstration of motivic colimits. We will provide the motivic analog of the following construction. Suppose that $\scr C$ is a semiadditive $\infty$-category, i.e., a pointed $\infty$-category with finite products and coproducts such that for any $X, Y \in \scr C$, the canonical map
\[
X \amalg Y \rightarrow X \times Y
\]
is an equivalence; we denote the corresponding biproduct by $X \oplus Y$. Any semiadditive $\infty$-category enjoys the following features \cite[Section 2.1]{DAG13}:
\begin{enumerate}
\item for any morphism of Kan complexes $f:T \rightarrow S$ with discrete fibers with finitely many components, the pullback functor $f^*: \scr C^S \rightarrow \scr C^T$ admits a right adjoint (resp. a left adjoint):
\[
f_*\,(\text{resp. } f_!): \scr C^T \rightarrow \scr C^S.
\]
\item There is a norm map
\[
\mathrm{Nm}_f: f_! \Rightarrow f_*,
\]
adjoint to an ambidextrous transformation
\[
f^*f_! \Rightarrow \id.
\]
The former transformation is furthermore invertible.
\item Suppose that $p: S \rightarrow *, q: T \rightarrow *$ be the canonical maps. Applying the functor $p_!$ to the transformation
\[
\id \rightarrow f_*f^* \simeq f_!f^*,
\]
we get a map
\[
tr_f:p_! \rightarrow p_!f_!f^* \simeq q_!f^*, 
\]
which is usually called the transfer along $f$. One of its key features is as follows: if $X \in \scr C^S$, then
\[
p_!X \rightarrow p_!f^*X \simeq p_!f_!f^*X \xrightarrow{p_!\eta} p_!X,
\]
is invertible after inverting the ``multiplication by $n$" endomorphism in $\scr C$ (which is defined because $\scr C$ is semiadditive).
\end{enumerate}

In this section, we will paint a motivic analog of the above picture; our main result, Corollary~\ref{cor:mot-colimits-transfers}, proves a motivic analog of the third property above. Given the ubiquity and utility of these constructions in classical homotopy theory (for example in the formulation of topological cyclic homology of Nikolaus and Scholze \cite{nikolaus-scholze}), we hope that our formalism will be independent interest.

Let us first work towards the motivic analog of semiadditivity. We call a morphism $f: \scr X \to \scr Y \in \PSh(\Sm_S)$ \emph{finite étale} (respectively a \emph{clopen immersion}) if it is representable by a finite étale morphism (respectively clopen immersion) of schemes, i.e., if for every $Y \in \Sm_S$ and every morphism $Y \to \scr Y$ the pullback $\scr X \times_{\scr Y} Y \to Y$ is isomorphic to a finite étale morphism (respectively clopen immersion) of schemes.
Note that if $f$ is finite étale (respectively a clopen immersion) then $\Delta_f: \scr X \to \scr X \times_{\scr Y} \scr X$ is a clopen immersion (respectively an isomorphism), since the same holds for schemes and colimits are stable under pullback in the $\infty$-topos $\PSh(\Sm_S)$.

\begin{remark}\label{rem:0-fet} It is possible to iterate the above definition in an analogous fashion to truncativity of maps in classical homotopy theory. From this point-of-view, a clopen morphism should be called \emph{$(-1)$-finite \'etale} while a finite \'etale morphism should be called \emph{$0$-finite \'etale}. A \emph{$1$-finite \'etale} morphsim should then be a morphism of stacks $\scr X \rightarrow \scr Y$ whose diagonal is finite \'etale. This notion is not very interesting if the stacks are schemes, but is already interesting for maps of Deligne-Mumford stacks (for example, $BG \rightarrow *$ for $G$ a finite \'etale group scheme). With this, one can develop a notion of higher $\FEt$-semiadditivity, which we postpone to a sequel.
\end{remark}

\begin{definition} \label{def:fet-semiadd}
Suppose that we are given $\widehat{\scr C}: \PSh(\Sm_S)^\op \to \widehat{\Cat}_\infty$.
\begin{enumerate}
\item We say that $\widehat{\scr C}$ \emph{has $\FEt$-colimits} if for every (relative) finite étale morphism $f: \scr X \to \scr Y$ the functor $f^*: \widehat{\scr C}(\scr Y) \to \widehat{\scr C}(\scr X)$ admits a left adjoint $f_\sharp$.
\item We say that the \emph{$\FEt$-colimits} are compatible with base change if for any cartesian square
\begin{equation*}
\begin{CD}
\scr X' @>g'>> \scr X \\
@V{f'}VV        @VfVV \\
\scr Y' @>g>>  \scr Y
\end{CD}
\end{equation*}
  with $f$ finite étale, the canonical exchange transformation $f'_\sharp g'^* \Rightarrow g^* f_\sharp$ is an equivalence.
\item Given a morphism $f: \scr X \to \scr Y$, form the square
\begin{equation*}
\begin{CD}
\scr X \times_{\scr Y} {\scr X} @>f_1>> \scr X \\
@V{f_2}VV        @VfVV \\
\scr X @>f>>  \scr Y
\end{CD}
\end{equation*}
Suppose that $\scr C$ has $\FEt$-colimits compatible with base change. We say that $\widehat{\scr C}$ is \emph{semiadditive} if for every clopen immersion $f: \scr X \to \scr Y$ the canonical, invertible transformation \[ f^* f_\sharp \wequi f_{2\sharp}f_1^* \xrightarrow{\wequi} \id \] exhibits $f_\sharp$ as right adjoint to $f^*$.
 
Here the first equivalence is because $\widehat{\scr C}$ has $\FEt$-colimits which are compatible with base change and the second is because $f$ is a clopen immersion, so the projections $f_{1,2}: \scr X \times_{\scr Y} \scr X \to \scr X$ are isomorphism.
\item We say that $\widehat{\scr C}$ is \emph{$\FEt$-semiadditive} if it is semiadditive and, in addition, for every finite étale morphism $f: \scr X \to \scr Y$ with diagonal $\Delta$, the canonical transformation 
\begin{equation} \label{eq:counit}
 f^*f_\sharp \wequi f_{2\sharp}f_1^* \Rightarrow f_{2\sharp}\Delta_\sharp \Delta^* f_1^* \wequi \id 
 \end{equation} exhibits $f_\sharp$ as right adjoint to $f^*$.
  Here we use that $\Delta$ is clopen, so $\Delta_\sharp$ is right adjoint to $\Delta^*$ by the semiadditivity assumption.
\end{enumerate}
We say that $\widehat{\scr C}$ is \emph{$\FEt$-semiadditive on schemes} if (4) holds whenever $\scr X, \scr Y \in \Sm_S$, and analogously for the other definitions.
\end{definition}

\begin{remark}\label{defn:norm} Suppose that $\widehat{\scr C}$ is $\FEt$-semiadditive and $f: \scr X \rightarrow \scr Y$ is a finite \'etale morphism, then the transformation~\eqref{eq:counit} is adjoint to a transformation
\[
\mathrm{Nm}^{\mot}_f:f_{\sharp} \rightarrow f_*
\]
This is a generalization of the norm construction for additive $\infty$-categories as in \cite[Section 2.1]{DAG13} and we call it the \emph{motivic norm transformation} along $f$.
\end{remark}

\begin{example}
$\SH(\ph)$ is $\FEt$-semiadditive.
Indeed the transformation $f^*f_\sharp \Rightarrow \id$ coincides (by construction) with the ambidexterity transformation \cite[(5.20)]{Hoyois:6functor}, which exhibits $f_\sharp$ as right adjoint to $f^*$ by \cite[Theorem 6.9]{Hoyois:6functor}. This property then persists for modules over any (highly structured) motivic ring spectrum such as $H\Z, KGL, MGL$ and the other usual suspects.
\end{example}

\begin{lemma} \label{lemm:continuous-extension-still-semiadd}
Suppose given $\scr C: \Sm_S^\op \to \widehat{\Cat}_\infty$ and write $\widehat{\scr C}$ for its continuous extension to $\PSh(\Sm_S)^\op$.
Suppose that $\scr C$ has $\FEt$-colimits stable under base change (respectively is semiadditive, respectively is $\FEt$-semiadditive), on schemes.
Then $\widehat{\scr C}$ has $\FEt$-colimits stable under base change (respectively is semiadditive, respectively is $\FEt$-semiadditive).
\end{lemma}
\begin{proof}
Recall that given a diagram $F: I^\op \to \widehat{\Cat}_\infty$, its limit can be modelled as the cartesian sections of the associated fibration \cite[Corollary 3.3.3.2]{HTT}: \[ \lim_I F \subset \Sect_{I}\left(I_{\sslash F}\right). \]
Suppose given a natural transformation $\alpha: F \to G \in \Fun(I^\op, \widehat{\Cat}_\infty)$ and suppose that each component $\alpha(i): F(i) \to G(i)$ admits a left adjoint, compatibly with the transition maps.
Then the associated morphism of fibrations is a relative right adjoint \cite[Lemma D.3, Remark D.4]{norms} and hence the induced cartesian morphism of cartesian fibrations has a left adjoint $\beta$ computed objectwise \cite[Lemma D.6, Remark D.4]{norms}.
Since the left adjoints of the $\alpha(i)$ are compatible, $\beta$ preserves the subcategory of cartesian sections.
Thus, under the above assumptions, $\lim_I \alpha: \lim_I F \rightarrow \lim_I G$ has a left adjoint computed in the obvious way in terms of the left adjoints of the $\alpha(i)$.

Now suppose given a finite étale morphism $f: \scr X \to \scr Y$. Since $\widehat{\scr C}$ is continuous we have \[ \widehat{\scr C}(\scr Y) \wequi \lim_{T \to \scr Y} \scr C(T) \quad\text{and}\quad \widehat{\scr C}(\scr X)  \wequi \lim_{T \to \scr Y} \scr C(T \times_{\scr Y} \scr X), \] where the limits are indexed by (the opposite of) $I = (\Sm_S)_{\sslash \scr Y}$.
Applying the above discussion, we deduce that $f^*: \widehat{\scr C}(\scr Y) \to \widehat{\scr C}(\scr X)$ admits a left adjoint $f_\sharp$ computed objectwise.
All other claims are immediate from this description.
\end{proof}

Note that the above proof shows that $\widehat{\scr C}(\scr X)$ can be identified with the (non-full) subcategory of $\PSh(\Sm_S)_{\sslash \scr C} \subset \PSh((\Sm_S)_{\sslash \scr C})$ where the source is identified with $\scr X$.
\begin{proposition} \label{prop:mot-kan}
Suppose that motivic colimits exist in $\scr C$.
The canonical functor $\scr C(S) \to \widehat{\scr C}(\scr X)$ (pullback along $p: \scr X \to *$) admits a left adjoint $p_\sharp$, given by \[ \widehat{\scr C}(\scr X) \to \PSh((\Sm_S)_{\sslash \scr C}) \xrightarrow{M} \scr C(S). \]
\end{proposition}
\begin{proof}\todo{I would prefer a more careful proof.}
It suffices to understand the left adjoint of \[ \tilde p^*: \scr C(S) \xrightarrow{p^*} \widehat{\scr C}(\scr X) \subset \Sect_{\scr X}(\scr C). \]
Let $\scr C_0$ be the constant functor at $\scr C(S)$.
Then $\tilde p^*$ factors as \[ \scr C(S) \xrightarrow{F} \Sect_{\scr X}(\scr C_0) \xrightarrow{G} \Sect_{\scr X}(\scr C). \]
Arguing as in the proof of Lemma \ref{lemm:continuous-extension-still-semiadd}, the functor $G$ has a left adjoint computed objectwise, and under the equivalence \[ \Sect_{\scr X}(\scr C_0) \wequi \Fun((\Sm_S)_{\sslash \scr X}^\op, \scr C(S)) \] the functor $F$ has as left adjoint the colimit functor.
Using Remark \ref{rmk:compute-motivic-colim-explicit}, the left adjoint of $\tilde p^*$ thus coincides with $M$.
\end{proof}

We call functors \[ p_{\sharp}:  \widehat{\scr C}(\scr X) \rightarrow \scr C(\scr Y) \] left adjoint to $p^*$ \emph{motivic left Kan extension functors} along $p$.
We have seen that $p_\sharp$ exists if either $p: \scr X \to \scr Y$ is finite étale and $\scr C$ has $\FEt$-colimits stable under base change (Lemma \ref{lemm:continuous-extension-still-semiadd}) or $\scr Y = S$ (Proposition \ref{prop:mot-kan}).
This terminology is justified by the characterization of left Kan extensions as in \cite[Proposition 4.3.3.7]{HTT}. It is possible to formulate a more general theory of motivic left Kan extensions, which we leave to future work. For now we have the following application. 

\begin{corollary} \label{cor:mot-colimits-transfers}
Let $\scr C$ be $\FEt$-semiadditive, $f: \scr X \to \scr Y \in \PSh(\Sm_S)$ be a finite étale morphism and $\scr Y \xrightarrow{D} \scr C$.
Then there is a canonical transfer morphism $\tr_f: M(D) \to M(Df)$.
Moreover, if $f$ is of degree $n$, and $Z \in \scr C(S)$ such that for every finite étale morphism $f_0: X \to Y \in \Sm_S$ of degree $n$ and every $E \in \scr C(Y)$ the canonical map $Z \otimes E \to Z \otimes f_{0\sharp}f_0^* E \to Z \otimes E$ is an equivalence, then so is $Z \otimes \tr_f$, provided that either $Z=\1$ or $\scr C$ satisfies the smooth projection formula.
\end{corollary}
\begin{proof}
We can view $D$ as an object of $\widehat{\scr C}(\scr Y)$ and similarly $Df$ as an object of $\widehat{\scr C}(\scr X)$; then $Df \wequi f^*(D)$.
By Lemma \ref{lemm:continuous-extension-still-semiadd} the functor $f^*$ has an ambidextrous adjoint $f_\sharp$, whence there is a canonical map \[ D \to f_\sharp f^* D \wequi f_\sharp Df. \]
Applying $p_\sharp$ and using Proposition \ref{prop:mot-kan} we obtain $\tr_f$.
Since a morphism in a section category is an equivalence if and only if it is so objectwise, the degree $n$ assumption implies that $Z \otimes D \to Z \otimes Df \to Z \otimes D$ is an equivalence; the moreover part follows.
\end{proof}

\begin{remark} Suppose that $\scr Y = Y, \scr X = X$ are (represented by) smooth schemes. Then $\alpha$ corresponds to an object $E \in {\scr C}(X)$, $f \circ \alpha$ corresponds to $f^* E$, and $f: g_\sharp f_\sharp E \to g_\sharp E$, where $g: X \to S$ is the structure map, is the canonical map. Then $tr_g: g_{\sharp} E \to g_{\sharp} f_{\sharp} E$ is supposed to be $g_{\sharp}$ of a canonical map $E \to f_{\sharp} f^*E$, namely the (finite) étale transfer (see e.g. \cite[Section 4]{bachmann-real-etale} \cite[Section 2.3]{ro-rigidity} for the case $\scr C(X) = \SH(X)$). Reverting to the general case, we can write $\scr X$ as a colimit of smooth schemes $X$, and then $\scr Y$ is the corresponding colimit of smooth $S$-schemes $f_X: f^{-1}(X) \to X$ finite étale over $X$. Informally, transfer $tr_f$ is the colimit of the transfers $tr_{f_X}$; the formalism of motivic colimits makes this precise.
\end{remark}


\section{Motivic extended and generalized powers} \label{sec:motivic-extended-powers}

In this section, we give the second working demonstration of motivic colimits by constructing motivic extended and generalized powers.
Fix a functor \[ \scr C^\otimes: \Span(\Sm_S, \fet, \all)^\op \wequi \Span(\Sm_S, \all, \fet) \to \widehat\Cat_\infty \] as in Section \ref{subsec:interaction-colimits-norms}.

For any scheme $X$, let us denote by $\FEt_X^\wequi$ the groupoid whose objects are finite \'etale morphisms $f: Y \rightarrow X$ and the morphisms are isomorphisms over $X$. The starting point for the construction of motivic extended and generalized powers is a functor of the form
 \[ N: \PSh(\Sm_S)_{/\FEt^\wequi} \times \scr C(S) \to \PSh(\Sm_S)_{\sslash \scr C}, \] where $\FEt^\wequi$ for the presheaf of groupoids 
\[
\FEt^\wequi: \Sm_S^\op \to \Spc, X \mapsto \FEt_X^\wequi.
\]
We will call this functor the \emph{fundamental diagram}. This functor is informally described as follows: it sends the pair
\[
(\scr X \to \FEt^\wequi, E \in \scr C(S))
\] to the diagram 
\[
\scr X \to \scr C \in \PSh(\Sm_S)_{\sslash \scr C},
\] which on sections over $Y \in \Sm_S$ is given by $\scr X(Y) \to \FEt_Y^\wequi \xrightarrow{N_E(Y)} \scr C(Y)$. Here, $N_E(Y): \FEt_Y^\wequi \to \scr C(Y)$ sends a finite \'etale morphism $p: U \to Y$ to $p_\otimes E_U$.

\subsection{The fundamental diagram}
\newcommand{\cart}{\mathrm{cart}}
\newcommand{\Cart}{\mathrm{Cart}}

Let us now make the above construction precise. Recall that given a small $\infty$-category $\scr D$, the \emph{Grothendieck construction} establishes an equivalence \cite[Section 3.2]{HTT}
\begin{equation} \label{eq:str}
\int: \Fun(\scr D^\op, \widehat{\Cat}_\infty) \simeq \Cart_{\scr D}.
\end{equation}
Here the right hand side has as objects the cartesian fibrations over $\scr D$.
If we have cartesian fibrations $\scr E_1, \scr E_2$ over $\scr D$, then the space of maps in the $\Cart_{\scr D}$ is defined as \[ \Map_{\scr D}^\cart(\scr E_1, \scr E_2) = \Fun_\scr{D}^{\cart}(\scr E_1, \scr E_2)^\wequi, \] where \[ \Fun_\scr{D}^{\cart}(\scr E_1, \scr E_2) \subset \Fun_\scr{D}(\scr E_1,\scr E_2) = \Fun(\scr E_1, \scr E_2) \times_{\Fun(\scr E_1, \scr D)} \{\scr E_1\} \] is the full subcategory on those functors preserving cartesian edges.

The next lemma follows once we know that the equivalence \eqref{eq:str} promotes to an equivalence of $(\infty,2)$-categories, though we do not need to use this language. To formulate it, consider the Yoneda functor
 \[
 R:\scr D \rightarrow \Fun(\scr D^\op, \Spc) \subset \Fun(\scr D^\op, \widehat{\Cat}_\infty).
 \]
We denote by $R_d \rightarrow \scr D$, the cartesian fibration obtained by performing the Grothendieck construction on $R(d)$, i.e., $R_d$ is the slice category $\scr D_{/d} \rightarrow \scr D^\op$.

\begin{lemma} \label{lemm:2-yoneda-bs}
For any $d \in \scr D$, evaluation at $(d,\id)$ defines an equivalence 
\[
\Fun^{\cart}_{\scr D}(R_d, \scr E) \wequi \scr E_d,
\]
where $\scr E_d$ is the fiber over $d$.
\end{lemma}

\begin{proof} This is a special case of \cite[Lemma 2.5.7]{e-rune} where $L$ is considered to be all morphsims.

\end{proof}
%
%

%

\begin{construction} \label{cons:N}
Consider the composite \begin{gather*} N_0^\dagger: \scr C(S) \wequi \Fun_{\Span(\Sm_S, \fet, \all)}^{\cart}(R_S, \Span(\Sm_S, \fet, \all)_{\sslash \scr C}) \\ \to \Fun_{\Sm_S}^{\cart}((\Sm_S)_{\sslash \FEt^\wequi}, (\Sm_S)_{\sslash \scr C}) \subset \Fun((\Sm_S)_{\sslash \FEt^\wequi}, (\Sm_S)_{\sslash \scr C}). \end{gather*}
Here the equivalence is given by Lemma \ref{lemm:2-yoneda-bs} and the map is restriction from fibrations over $\Span(\Sm_S, \fet, \all)$ to fibrations over $\Sm_S$, using that the functor represented by $S$ in $\Span(\Sm_S, \fet, \all)$ is exactly $\FEt^\wequi$.
By adjunction this corresponds to \[ N_0: (\Sm_S)_{\sslash \FEt^\wequi} \times \scr C(S) \to (\Sm_S)_{\sslash \scr C}. \]
The composite \[ N: \PSh(\Sm_S)_{/\FEt} \times \scr C(S) \to \PSh((\Sm_S)_{\sslash \FEt^\wequi} \times \scr C(S)) \xrightarrow{\PSh(N_0)} \PSh((\Sm_S)_{\sslash \scr C}) \] takes values in $\PSh(\Sm_S)_{\sslash \scr C}$ (indeed this is true for $N_0$, and $\PSh(\Sm_S)_{\sslash \scr C} \subset \PSh((\Sm_S)_{\sslash \scr C})$ is closed under colimits\todo{why?}).
\end{construction}

\begin{remark} \label{rmk:N-cocont}
By construction, $N$ is a cocontinuous functor in the first variable.
Hence for $(\alpha: \scr X \to \FEt, E) \in \PSh(\Sm_S)_{/\FEt} \times \scr C(S)$ we have \[ N(\alpha, E) = \colim_{p' \in (\Sm_S)_{\sslash \scr X}} (T, \alpha(p')_\otimes(E_U)), \] where we write $\alpha(p'): U \to T \in \FEt$.
\end{remark}

For any scheme $S$, we define the \emph{fundamental diagram} as the functor
\[
 \PSh(\Sm_S)_{/\FEt^\wequi} \times  \scr C(S) \xrightarrow{N} \PSh(\Sm_S)_{\sslash \scr C}.
\]
As explained in the next section, the motivic extended powers will be constructed out of the fundamental diagram.

\subsection{The definition of motivic extended powers} We now use the fundamental diagram functor $N$ to define motivic extended and genearlized powers, by taking appropriate motivic colimits. In preparation of that, we denote by $\FEt^{\wequi,n} \subset \FEt$ the subpresheaf of groupoids spanned by finite étale schemes of rank $n$.
\begin{lemma}\label{lem:stk} For any scheme $S$, we have a canonical equivalence
\[
\FEt^{\wequi, n}|_{\Sm_S} \wequi B_\et \Sigma_n \in \PSh(\Sm_S).
\]
 \end{lemma}
\begin{proof}
This follows from the standard correspondence between finite \'etale morphisms of degree $n$ and $\Sigma_n$-torsors \cite[(6.1.3)]{champs}.\todo{ref}
\end{proof}

\begin{definition}
Suppose that $\scr C(S)$ admits small colimits. We define the \emph{generalized motivic extended power} functor as the composite
\begin{equation} \label{eq:gen-power}
 D_{gen}^\mot: \PSh(\Sm_S)_{/\FEt^\wequi} \times  \scr C(S) \xrightarrow{N} \PSh(\Sm_S)_{\sslash \scr C} \xrightarrow{M} \scr C(S). 
 \end{equation}
We adopt the following notation for evaluating the above bifunctor on on objects of $\PSh(\Sm_S)_{/\FEt^\wequi}$:
\begin{enumerate}
\item For $\scr X \in \PSh(\Sm_S)_{/\FEt^\wequi}$ fixed, we denote the functor $D_{gen}^\mot(\scr X, \ph)$ by 
\[
D_\scr{X}^\mot:  \scr C(S) \rightarrow  \scr C(S);
\] 
\item if $\scr X = \FEt^{\wequi, n}$, we denote $D_\scr{X}^\mot$ by 
\[
D_n^\mot:  \scr C(S) \rightarrow  \scr C(S);
\] 
\item if $G \to \Sigma_n$ is a group homomorphism, and $\scr X \in \PSh(\Sm_S)_{/\FEt^\wequi}$ is given by the composite $B_\et G \to B_\et \Sigma_n \wequi \FEt^{\wequi, n} \to \FEt$, then we denote $D_\scr{X}^\mot$ by 
\[
D_G^\mot:\scr C(S) \rightarrow  \scr C(S);
\]
\item if $BG$ is the constant presheaf, then there is a canonical map $BG \to B\Sigma_n \to B_\et \Sigma_n$, and we put 
\[
D^\mot_{BG} =: D^\naive_{BG}:\scr C(S) \rightarrow  \scr C(S),
\] and in particular \[
D^\mot_{B\Sigma_n} =: D^\naive_n:\scr C(S) \rightarrow  \scr C(S).
\]
\end{enumerate}
\end{definition}

\begin{example} \label{ex:D-gen-colimit}
Since the motivic colimit functor is cocontinuous, it follows from Remark \ref{rmk:N-cocont} that \[ D^\mot_{\scr X}(E) \wequi \colim_{p' \in (\Sm_S)_{\sslash \scr X}} p_\otimes E_U, \] where $\alpha(p') = p: U \to X$.
\end{example}

\begin{example} \label{ex:D-naive-mot}
It follows from Example \ref{ex:constant-presheaves} that $D^\naive_G E$ is given by the ordinary homotopy $G$-orbits of the $G$-action on $E^{\otimes n}$.
Since $D^\mot_{gen}$ is functorial in $\scr X$, we obtain natural transformations $D^\naive_G \Rightarrow D^\mot_G$.
\end{example}

\begin{example} \label{ex:D-mot-prod}
We have $D^\mot_{H_1 \times H_2}(E) \wequi D^\mot_{H_1}(E) \wedge D^\mot_{H_2}(E)$. Indeed this is a consequence of Proposition \ref{prop:M-p-otimes}, noting that $B_\et(H_1 \times H_2) \wequi B_\et(H_1) \times B_\et(H_2)$.
\end{example}

\subsection{Basic properties} We now explain how the construction above interacts with various categorical operations.

\subsubsection{Extended powers and colimits} 

\begin{proposition} \label{prop:D-mot-colim}
\begin{enumerate}
\item The functor $D^\mot_{gen}: \PSh(\Sm_S)_{/\FEt^\wequi} \times \scr C(S) \to \scr C(S)$ preserves colimits in the factor $\PSh(\Sm_S)_{/\FEt^\wequi}$.
\item If each of the functors $f^*, p_\otimes$ preserves colimits of some shape $K$, then $D^\mot_{gen}$ preserves colimits of shape $K$ in the factor $\scr C(S)$.
\end{enumerate}
\end{proposition}
\begin{proof}
(1) Clear from cocontinuity of motivic colimits and Remark  \ref{rmk:N-cocont}.

(2) Since colimits commute, this follows from Example \ref{ex:D-gen-colimit}.
\end{proof}

\begin{example}
In each of the cases from Example \ref{ex:changing-C}, the functors $p_\otimes, f^*$ preserve sifted colimits, and hence so does $D^\mot_\scr{X}$.
\end{example}

\begin{example} \label{ex:fet-decomp}
If $\scr C: \Sm_S^\op \to \widehat{\Cat}_\infty$ preserves finite products (e.g. in each of the cases from Example \ref{ex:changing-C}), we have
\[ D^\mot_{\FEt^\wequi}(E) \wequi \bigvee_{n \ge 0} D^\mot_n(E). \]
Indeed $\coprod_{n \ge 0} \FEt^{\wequi, n} \to \FEt^\wequi$ is an $L_\Sigma$-local equivalence (see Corollary~\ref{cor:stk-fet}) and $M$ factors through $L_\Sigma$-local equivalences by Example \ref{ex:M-P-Sigma}.
\end{example}

\subsubsection{Changing $\scr C$} Often, we consider two functors $\scr C$ and $\scr D$ out of $\Span(\Sm_S, \all, \fet)$ and a transformation $\scr C \Rightarrow \scr D$ between them. For example $\scr C = \SH$ and $\scr D = \DM$ and the transformation given by associating to a motivic spectrum its motive, i.e., the free $H\Z$-module on it. Other examples include the localization functor with respect to some topology such $L_{\et}:\SH \rightarrow \SH_{\et}$.

\begin{proposition} \label{prop:Dmot-changing-C}
\NB{Start more generally with $\eta: \Span_{\sslash \scr C} \to \Span_{\sslash \scr D}$?}
Suppose given $\eta: \scr C^\otimes \to \scr D^\otimes \in \Fun(\Span(\Sm_S, \all, \fet), \widehat{\Cat}_\infty)$.
\begin{enumerate}
\item There is a canonical natural transformation
\[ ex_{N \eta}: \bar{\eta} \circ N^\scr C \Rightarrow N^\scr{D} \circ (\ph, \eta_S): \PSh(\Sm_S)_{/\FEt^\wequi} \times \scr C(S) \to \PSh(\Sm_S)_{\sslash \scr D}. \]
\item If the assumptions of Proposition \ref{prop:changing-C} (for $ex_{M\eta}$ to be an isomorphism) are satisfied, then there is an induced natural transformation \[ex_{D^\mot \eta}: \eta_S D^\mot_{gen} \Rightarrow D^\mot_{gen}(\ph, \eta_s): \PSh(\Sm_S)_{/\FEt^\wequi} \times \scr C(S) \to \scr D(S).\]
\item Under the assumptions of (2), the transformation $ex_{D^\mot \eta}$ is an isomorphism.
\end{enumerate}
\end{proposition}
\begin{proof}
(1) All steps of Construction \ref{cons:N} are natural in $\scr C$, whence we obtain the desired transformation.

(2) By Proposition \ref{prop:changing-C} and (1) we have a span of natural transformations
\[ \eta_S D^\mot_{gen} = \eta_S M^\scr{C} N^\scr{C} \xLeftarrow{ex_{M\eta}} M^\scr{D} \bar{\eta} N^\scr{C} \xRightarrow{ex_{N\eta}} M^\scr{D} N^\scr{D} \circ (\ph, \eta_S) = D^\mot_{gen} \circ (\ph, \eta_S) \]
and the first one is a natural isomorphism. This constructs $ex_{D^\mot \eta}$.

(3) Since $\eta$ commutes with $p_\otimes, f^*$ and $\eta_S$ preserves colimits, this follows from Example \ref{ex:D-gen-colimit}.
\end{proof}

\begin{example} \label{ex:Dmot-commutes-with-suspension}
In each of the cases from Example \ref{ex:changing-C}, the functor $\eta_S$ commutes with the functor $D^\mot_{gen}$.
\end{example}

\begin{example} \label{ex:Dmot-betti-realisation}
Let $r_\C: \SH(\C) \to \SH$ denote the Betti realisation. Then $r_\C D^\mot_G(E) \wequi D_G(r_\C E)$. Indeed since $r_\C$ factors through $L_\et: \SH(\C) \to \SH^\et(\C)$, by Proposition \ref{prop:Dmot-changing-C} it suffices to deal with $E \in \SH^\et(\C)$. But then $BG \to L_\et BG \wequi B_\et G$ is an étale-local weak equivalence, so \[ D^\mot_G(E) = D^\mot_{B_\et G}(E) \wequi D^\mot_{BG}(E) \wequi D^\naive_G(E), \] by Corollary \ref{corr:M-local-equiv} and Example \ref{ex:constant-presheaves}. Since $r_\C$ is symmetric monoidal and preserves colimits, it commutes with $D^\naive_G$ as needed. The same argument applies to $r_\C: \Spc(\C)_\pt \to \Spc$ and $r_\C: \Spc(\C) \to \Spc$.
\end{example}

\subsubsection{Changing $S$} We begin to address the interaction of motivic colimits with base change.

\begin{proposition} \label{prop:D-mot-changing-S}\NB{Assumptions not quite minimal.}
Suppose given $\scr C^\otimes: \Span(\Sch_S, \all, \fet) \to \widehat{\Cat}_\infty$ and $f: T \to S \in \Sch_S$.
\begin{enumerate}
\item There is a canonical natural isomorphism \[ ex_N^*: N_T \circ (f^*, f^*) \Rightarrow f^* N_S \in \Fun(\PSh(\Sm_S)_{/\FEt^\wequi} \times \scr C(S), \PSh(\Sm_T)_{\sslash \scr C_T}). \]
\item If the restriction of $\scr C$ to $\Sch_S^\op$ satisfies the assumptions of Proposition \ref{prop:changing-S}(1), then there is an induced natural transformation $ex_{D^\mot}^*: D^\mot_{gen} \circ (f^*, f^*) \Rightarrow f^* D^\mot_{gen}$.
\item If the restriction of $\scr C$ to $\Sch_S^\op$ satisfies the assumptions of Proposition \ref{prop:changing-S}(2), then $ex_{D^\mot}^*$ is an isomorphism.
\end{enumerate}
\end{proposition}
\begin{proof}
(1)\tombubble{This is messier than I would like.} By taking a cocontinuous extension, it suffices to construct the transformation of functors $(\Sm_S)_{/\FEt} \times \scr C(S) \to (\Sm_T)_{\sslash \scr C}$.
It can be informally described as \[ ((X, p \in \FEt_X), E \in \scr C(S)) \mapsto ((p  \times_S T)_\otimes(E|_{X \times_S T}) \xrightarrow{\wequi} f^*p_\otimes(E|_X)). \] In particular the resulting transformation is a natural isomorphism, as claimed.
To do this, we may as well construct the transformation of functors on the larger categories $(\Sch_S)_{/\FEt} \times \scr C(S) \to (\Sch_T)_{\sslash \scr C}$.
The functor $f^*: (\Sch_S)_{\sslash \scr C} \to (\Sch_T)_{\sslash \scr C}$ has a left adjoint $f_\sharp$ (see Construction \ref{f*-construct}).
It is thus enough to construct \[ f_\sharp N_0^T f^* \Rightarrow N^0_S: (\Sch_S)_{/\FEt} \times \scr C(S) \to (\Sch_S)_{\sslash \scr C}, \] or equivalently by adjunction a transformation of functors \[ \scr C(S) \wequi \Fun_{\Span(\Sch_S, \fet, \all)}^\cart(R_S, (\Sch_S)_{\sslash \scr C}) \to \Fun((\Sch_S)_{\sslash \FEt^\wequi}, (\Sch_S)_{\sslash \scr C}). \]
The target functor is just $N_0^\dagger$, i.e. restriction along $\Sch_S \to \Span(\Sch_S, \fet, \all)$ followed by the natural inclusion.
Unwinding the definitions, the source functor is obtained by first restricting along $\Span(\Sch_T, \fet, \all) \to \Span(\Sch_S, \fet, \all)$, then restricting along $\Sch_T \to \Span(\Sch_T, \fet, \all)$, postcomposing with $f_\sharp: (\Sch_T)_{\sslash \scr C} \to (\Sch_S)_{\sslash \scr C}$ and precomposing with $f^*: (\Sch_S)_{\sslash \FEt^\wequi} \to (\Sch_T)_{\sslash \FEt^\wequi}$.
We may symbolically depict this as \[ \Fun_{\Span(\Sch_S, \fet, \all)}^\cart(R_S, (\Sch_S)_{\sslash \scr C}) \ni F \mapsto f_\sharp \circ F|_{\Sch_T} \circ f^*. \]
As can be verified on the level of simplicial sets, this is the same as \[ F \mapsto F \circ f_\sharp \circ f^*. \]
Thus the desired transformation is obtained by composition with the unit transformation of \[ f_\sharp: (\Sch_T)_{\sslash \FEt^\wequi} \adj (\Sch_S)_{\sslash \FEt^\wequi}: f^*. \]

(2, 3) We have the composite of natural transformations
\[ D^\mot_{gen} \circ (f^*, f^*) \wequi M_T N_T \circ (f^*, f^*) \xRightarrow{M_T ex_N^*} M_T f^* N_S \xRightarrow{ex_M^* N_S} f^* M_S N_S \wequi f^* D^\mot_{gen}, \]
where the first one is obtained by (1) and is an isomorphism, and the second one is obtained by Proposition \ref{prop:changing-S}(1). We have thus obtained $ex_{D^\mot}^*$. It is an isomorphism as soon as $ex_M^*$ is, which is ensured by Proposition \ref{prop:changing-S}(2).
\end{proof}

\begin{example}\label{exam:dmot-f}
In each of the cases from Example \ref{ex:changing-C}, $D^\mot_{gen}$ commutes with arbitrary base change in the sense that $f^* D^\mot_{gen}(\scr X, E) \wequi D^\mot_{gen}(f^* \scr X, f^* E)$ for any $f, \scr X, E$.
\end{example}

Suppose that we have a morphism $f: T \to S$. The equivalence $f^*D^{\mot}_n(E) \simeq D^\mot_{gen}(f^*B_{\et}\Sigma_n, f^* E)$ from Example~\ref{exam:dmot-f}, and the comparison map of presheaves 
\[
f^*B_{\et}\Sigma_n|_{\Sm_S} \rightarrow B_{\et}\Sigma_n|_{\Sm_T}
\] defines a comparison morphism
\begin{equation}\label{eq:powers-compare}
f^*D^{\mot}_n(E) \rightarrow D^{\mot}_n(f^*E).
\end{equation}
It is not clear that this is an equivalence, because formation of $B_\et \Sigma_n$ need not commute with $f^*$ in general.
Using Example \ref{ex:M-tau-local} we see that this problem goes away as soon as $f^* B_\et \Sigma_n \to B_\et \Sigma_n$ is a local equivalence in an appropriate topology. 


\begin{remark} \label{rmk:Dn-smooth-base-change}
If $f: T \to S$ is smooth, then the map $\alpha: f^* B_\et G \to B_\et f^* G \in \PSh(\Sm_T)$ from Corollary \ref{corr:Bet-stable-by-base-change} is a plain equivalence (as in established in the proof of that result). In particular $D_n^\mot$ etc. commute with smooth base change as soon as Proposition \ref{prop:changing-S} applies.
\end{remark}

\begin{example} \label{ex:Dn-stable-base-change}
Suppose that $\scr C$ is a sheaf in the Zariski topology. Then Corollary \ref{corr:Bet-stable-by-base-change}, Example \ref{ex:M-tau-local} and Proposition \ref{prop:D-mot-changing-S} imply that $D_n^\mot$, $D_G^\mot$ and $D_{\FEt^\wequi}^\mot$ commute with base change.
\end{example}

\begin{example} \label{ex:Dn-stable-base-change-specific}
Example \ref{ex:Dn-stable-base-change} applies in particular to each of the cases from Example \ref{ex:changing-C}.
\end{example}

\subsection{Example: the case of étale sheaves}
In this subsection we return to the setting of Section \ref{subsec:M-PSh}. In other words we put $\scr C^\otimes(X) = \PSh(\Sm_X)$. For $p: X \to Y$ finite étale, the norm $p_\otimes: \PSh(\Sm_X) \to \PSh(\Sm_Y)$ is just given by $p_*$. In particular, both norm and pullback preserve arbitrary colimits. We now compute $D^\mot_G(\scr X)$ for $\scr X \in \PSh(\Sm_S)$, and $G \to \Sigma_n$ some finite group. Since $D^\mot$ preserves arbitrary colimits in our situation (by Proposition \ref{prop:D-mot-colim}), this reduces to the case where $\scr X$ is \emph{representable}. Since representable presheaves are in fact étale sheaves, the following result essentially achieves what we want.

\begin{proposition} \label{prop:D-mot-etale}
Let $\scr X \in \PSh(\Sm_S)$ be an étale sheaf and $G \to \Sigma_n$ a finite group. Then there is a canonical equivalence
\[ D^\mot_G(\scr X) \wequi (\scr X^n)_{hG}^\et. \]
\end{proposition}
\begin{proof}
By the right hand side we mean the following: we take the presheaf $\scr X^n \in \PSh(\Sm_S)$ which comes equipped with the $G$ act on it. Then we take homotopy orbits, and take associated étale sheaf; this models homotopy orbits internal to an appropriate $\infty$-topos. Let $U \in \Sm_S$. We shall compute the sections over $U$ of both the left and the right hand side and show that they are equivalent, naturally in $U$. 

For the left hand side, we use Corollary \ref{corr:M-compute} to obtain
\begin{align*}
D^\mot_G(\scr X)(U) &\wequi \colim_{\tilde p \in B_\et G(U)} N_\scr{X}(p)(U) \\
                    &\wequi \colim_{\tilde p} (p_\otimes \scr X_E)(U) \\
                    &\wequi \colim_{\tilde p} (p_* \scr X_E)(U) \\
                    &\wequi \colim_{\tilde p} \Map(E, \scr X).
\end{align*}
Here we let $p: E \to U$ be the associated finite étale scheme of rank $n$ to the principal $G$-bundle $\tilde{p}$.

For the right hand side, we work in the $\infty$-topos $L_\et \PSh(\Sm_S)$ of étale sheaves on $\Sm_S$. Let $\scr X^n \sslash G \in L_\et \PSh(\Sm_S)$ denote the homotopy orbits computed in this topos. By the construction of a quotient stack (see \cite[Tag 0370]{stacks} for the 1-truncated case, and \cite[Section 3.2]{principal-nss} for a discussion on the formation of quotients\NB{better reference?}) we have
\[ \Map(U, \scr X^n \sslash G) \wequi \colim_{\tilde p: \tilde{E} \to U \in B_\et G(U)} \Map^G(\tilde E, \scr X^n). \]
After an identification between the indexing diagrams of the two colimits using Lemma~\ref{lem:stk}, it thus remains to construct a natural equivalence $\Map^G(\tilde E, \scr X^n) \xrightarrow{\alpha_*} \Map(E, \scr X)$. We do this as follows. Let $\ul{n}$ denote the set $\{1, \dots, n\}$ with its canonical $G$-action. The map $\alpha: \scr X^n \times \ul{n} \to \scr X, ((x_1, \dots, x_n), i) \mapsto x_i$ is $\Sigma_n$-equivariant (for $\Sigma_n$ acting trivially on the right hand side), and hence $G$-equivariant. Given $\beta \in \Map^G(\tilde E, \scr X^n)$, we have $\alpha_*(\beta) := \alpha \circ (\beta \times \ul{n}) \in \Map^G(\tilde E \times \ul{n}, \scr X) \wequi \Map(\tilde E \times_G \ul{n}, \scr X)$. Noting that $E = \tilde E \times_G \ul{n}$, $\alpha_*$ has the correct shape. Clearly $\alpha_*$ is natural. In order to prove that $\alpha_*$ is an equivalence, we use that both sides are (étale) sheaves as functors of $U$ (since $\scr X$ is), and hence we may assume that $\tilde E$ is split. In this case, the map is $\Map^G(U' \times G, \scr X^n) \to \Map(U' \times \ul{n}, \scr X)$, which is an equivalence.
\end{proof}

\NB{This result allows us to describe the functors $D_G^\mot: \Spc(S) \to \Spc(S)$ and, via a slight variant, $D_G^\mot: \Spc(S)_\pt \to \Spc(S)_\pt$. E.g. the latter is induced by the composite $\SmQP_+ \to L_\et \PSh(\Sm_S)_\pt \xrightarrow{\bullet^{\wedge n}\sslash G}  L_\et \PSh(\Sm_S)_\pt \to \Spc(S)_\pt$ by sifted-cocontinuity.}
\NB{Morel-Voevodsky \cite[Section 4.2]{A1-homotopy-theory} almost prove the following result: if $X$ is a smooth scheme with $G$-action, then $X \sslash G$ (étale homotopy orbits) is motivically equivalent to $E_{gm} G \times^{\et}_G X$. This allows for a more geometric identification of the functors above.}

\subsection{Further properties} We now expand on more subtle properties of the motivic extended and generalized power constructions, with applications to power operations in mind. 

\subsubsection{Iterated extended powers}

A key property useful in the computation with power operations are the so-called \emph{Adem relations}. They are analogous to the Adem relations in the more familiar Steenrod algebra. The Adem relations, as explained in the context of spaces by Dyer--Lashof \cite[Section 1]{dyer-lashof}, come from studying the relationship between iterated extended powers and generalized powers for wreath products of various groups. We provide the motivic analogs in this subsection. More precisely, we compute expressions like $D^\mot_n D^\mot_m (E)$ in terms of generalized powers for wreath products of groups.

Let us briefly recall how wreath products work. In general, if $G$ is a group acting on $\{1, 2, \dots, n\}$ (in other words we are provided with a homomorphism $G \to \Sigma_n$) and $H$ is any group, then we denote by $H \wr G$ the semidirect product $H^n \rtimes G$. In particular we have a split exact sequence of groups \begin{equation} \label{eq:wr-prod-exact-seq} 1 \to H^n \to H \wr G \to G \to 1. \end{equation} If $H$ acts on $\{1, \dots, m\}$, then $H \wr G$ acts on $\{1, \dots, n\} \times \{1, \dots, m\}$.

To make sense of the next lemma, recall that essentially by definition $B_\et G \in \PSh(\Sm_S)$ is the restriction of the quotient stack $[*/G]$ to $\Sm_S$; in other words, $B_\et G$ classifies principal $G$-bundles on smooth $S$-schemes. A homomorphism $f: G \to \Sigma_n$ induces $B_\et G \to B_\et \Sigma_n \to \FEt^\wequi$. Thus, for each principal $G$-bundle, we obtain an \emph{underlying finite étale scheme of rank $n$} (depending on $f$) via Lemma~\ref{lem:stk}. Let us also recall that if we have an exact sequence of groups $1 \rightarrow K \rightarrow G \rightarrow H \rightarrow 1$, then we have a fiber sequence in $\PSh(\Sm_S)$:

\begin{equation}\label{eq:fiber}
B_{\et}K \rightarrow B_{\et}G \rightarrow B_{\et}H,
\end{equation}
obtained by pullback of the map classifying the trivial $H$-bundle $S \rightarrow B_{\et}H$ along $B_{\et}G \rightarrow B_{\et}H$.

\begin{lemma} \label{lemm:bundle-technical}
Let $G$ be a finite group acting on $\{1, 2, \dots, n\}$ and $H$ any finite group. Let $f: X \to S \in \Sm_S$ and $\tilde{p}: X \to B_\et G$ classify a principal $G$-bundle with underlying finite étale scheme $p: F \to X$. Then there is a canonical cartesian square
\begin{equation} \label{eq:B-wr-cart-square}
\begin{CD}
L_\Zar f_\sharp p_*p^*B_\et H @>>> B_\et H \wr G \\
@V{\gamma}VV                        @VVV      \\
X                @>{\tilde{p}}>> B_\et G.
\end{CD}
\end{equation}
\end{lemma}
\begin{proof}\tombubble{This is messier than I would like.}
Let us first prove the equivalent result in $\PSh(\Sch_S)$.
Working directly from the definitions, the fiber $F = X \times_{[*/G]} [*/(H \wr G)]$ is the presheaf over $X$ taking $T \in \Sch_X$ to the groupoid of $H$-torsors on $F_T$.
This is exactly $f_\sharp^{b} p_*p^* B_\et H$, where the superscript $b$ reminds us that we are working in the big topos.\NB{The problem is that $\PSh(\Sm_S)_{/T} \wequi \PSh((\Sm_S)_{/T}) \ne \PSh(\Sm_T)$.}

To prove the actual result, since we are working Zariski locally, we may assume that $S$ and $X$ are affine and work in $\PSh(\SmAff_S)$.
It will be enough to prove that $f_\sharp p_*p^*B_\et H \wequi (f_\sharp^b p_*p^*[*/H])|_{\SmAff_S}$.
This is true when evaluating on $T \in \SmAff_X$.
Since the left hand side is a colimit of such $T$'s, it suffices to prove that the right hand side also is.
In other words it suffices to show that $p_* [*/H] \in \PSh(\Aff_X)$ is a colimit of smooth affine schemes, or equivalently is left Kan extended from $\SmAff_X$.
Via \cite[Proposition A.0.4]{EHKSY3} this follows from the fact that $[*/H]$ is a smooth stack with quasi-affine diagonal and that the Weil restriction $p_*$ preserves such stacks; see the proof of Lemma \ref{lemm:LKE-primitive} for more details.
\end{proof}

The next proposition is our main resulting concerning iterated extended powers. For the rest of this section, we fix a functor \[
\scr C^\otimes: \Span(\Sm_S, \all, \fet) \to \widehat{\Cat}_\infty,
\] satisfying the following properties:
\begin{itemize}
\item Each category $\scr C(X)$ is stable and admits small colimits.
\item $\scr C$ satisfies smooth base change. 
\item $\scr C^\otimes$ satisfies the distributivity law of \cite[Proposition 5.12]{norms}.
\item Each functor $f_\otimes$ preserves sifted colimits.
\end{itemize}
Under these assumptions, Propositions \ref{prop:changing-S} and \ref{prop:M-p-otimes} apply. 

\begin{proposition} \label{prop:Dmot-iterated-wreath}
Let $G \to \Sigma_n$ and $H \to \Sigma_m$ be group homomorphisms, with $G, H$ finite. Let $H \wr G \to \Sigma_{nm}$ correspond to the canonical action of $H \wr G$ on $\{1, \dots, n\} \times \{1, \dots, m\}$. Suppose further that $\scr C$ is a Zariski sheaf. Then there is a canonical equivalence of functors $D^\mot_G \circ D^\mot_H \wequi D^\mot_{H \wr G}$.
\end{proposition}
\begin{proof}
\NB{Make this more formal and less of a mess?}
Let $E \in \scr C(S)$. Consider the diagram $N: B_\et H \wr G \to \scr C$ with $M(N) = D^\mot_{H \wr G}(E)$. Let $q: B_\et H \wr G \to B_\et G$ be the canonical map. For $f: X \to S \in \Sm_S$ and $\tilde{p}: X \to B_\et G$ we have an induced diagram $N_{\tilde{p}}: q^{-1}(\{\tilde p\}) \to B_\et H \wr G \to \scr C$. By Lemma \ref{lemm:bundle-technical}, we have $q^{-1}(\{\tilde p\}) \wequi L_\Zar f_\sharp p_* p^*B_\et H$, where $p: F \to X$ is the finite étale scheme (of rank $n$) corresponding to $\tilde{p}: X \to B_\et G$. In other words \begin{equation} \label{eq:iterated-t0} N_{\tilde{p}} = L_\Zar f_\sharp p_\otimes N_p^0,\end{equation} where $N_p^0 \in \PSh(\Sm_F)_{/\scr C_T}$ is the diagram computing $D^\mot_H(E_F)$.
In particular \begin{equation} \label{eq:iterated-t1} N \wequi \colim_{F \xrightarrow{\tilde p} X \xrightarrow{f} S} L_\Zar f_\sharp p_\otimes N_p^0. \end{equation}

Let $N_p' \to N_p^0$ be the universal approximation of $N_p$ by smooth quasi-projective schemes, i.e., the left Kan extension of its restriction to smooth quasi-projective schemes. Since $N_p$ is a sheaf, it preserves finite products, whence so does $N_p'$; i.e. $N_p'$ is a sifted colimit of smooth quasi-projective schemes. Moreover $N_p' \to N_p$ is an equivalence on smooth quasi-projective schemes, so in particular a Zariski equivalence. Since $\scr C$ is a Zariski sheaf, by Example \ref{ex:M-tau-local} ($M$ preserves Zariski equivalences) and Proposition \ref{prop:M-p-otimes} (compatibility of $M$ with $p_\otimes$) we get \[ M_X(p_\otimes N_p^0) \wequi M_X(p_\otimes N_p') \wequi p_\otimes M_F(N_p') \wequi p_\otimes M_F(N_p^0) \wequi p_\otimes D^\mot_H(E_F). \] 

With the hypotheses on $\scr C$, we can use Example \ref{ex:M-tau-local} ($M$ preserves Zariski equivalences), Proposition \ref{prop:interaction-M-f-sharp} (compatibility of $M$ with $f_\sharp$) and Remark \ref{rmk:Dn-smooth-base-change} (compatibility of $D^\mot_H$ with smooth base change), to obtain equivalences: \begin{equation} \label{eq:iterated-t2} M_S(L_\Zar f_\sharp p_\otimes N_p^0) \wequi f_\sharp p_\otimes D^\mot_H(E_F) \wequi f_\sharp p_\otimes (D^\mot_H(E)_F). \end{equation}
The claim then follows by the following computation:
\begin{align*}
  D^\mot_G(D^\mot_H (E)) &\stackrel{E.\ref{ex:D-gen-colimit}}{\wequi} \colim_{F \xrightarrow{\tilde p} X \xrightarrow{f} S} f_\sharp p_\otimes (D^\mot_H (E)_F) \\
                         &\stackrel{\eqref{eq:iterated-t2}}{\wequi} \colim M_S(L_\Zar f_\sharp p_\otimes N_p^0) \\
                         &\stackrel{\eqref{eq:iterated-t0}}{\wequi} \colim M_S (N_{\tilde p}) \\
                         &\wequi M_S(\colim N_{\tilde p}) \\
                         &\stackrel{\eqref{eq:iterated-t1}}{\wequi} M_S (N) \\
                         &\wequi D^\mot_{H \wr G}(E).
\end{align*}
\end{proof}

\subsubsection{Excisivity results} With the assumptions on $\scr C$ from the previous section still in play, we discuss some excisivity properties of the extended powers. Recall that a functor $f: \scr D \to \scr E$ is called $n$-excisive if it sends strongly cocartesian $(n+1)$-cubes to cartesian $(n+1)$-cubes \cite[Definition 6.1.1.3]{HA}.

\begin{lemma} \label{lemm:f-otimes-excisive}
Let $f: X \to Y \in \Sm_S$ be finite étale of degree $\le n$. Then $f_\otimes: \scr C(X) \to \scr C(Y)$ is polynomial of degree $\le n$, and so $n$-excisive.
\end{lemma}
\begin{proof}
This is the same argument as \cite[Proposition 5.25]{norms}. For the $n$-excisivity, see \cite[Remark 5.22]{norms}.
\end{proof}

\begin{corollary} \label{corr:D-mot-excisive}
Let $\scr X \in \PSh(\Sm_S)_{/B_\et \Sigma_n}$. Then $D^\mot_\scr{X}: \scr C(S) \to \scr C(S)$ is $n$-excisive.
\end{corollary}
\begin{proof}

In a stable category any colimit of $n$-excisive functors is $n$-excisive \cite[Remark 6.1.5.10]{HA}. From the colimit formula for $D^\mot_\scr{X}$ in Example~\ref{ex:D-gen-colimit}, it thus remains to show that $f_\sharp g_\otimes p^*$ is $n$-excisive, for $g$ finite étale of degree $n$. This follows from Lemma \ref{lemm:f-otimes-excisive}, since $f_\sharp, p^*$ are exact.

\end{proof}

We also have the following observation. Recall that a functor is called reduced if it preserves the terminal object.

\begin{lemma} \label{lemm:f-otimes-reduced}
Let $f: X \to Y \in \Sm_S$ be finite surjective étale. Then $f_\otimes$ is reduced.
\end{lemma}
\begin{proof}
For $f: X \to Y$ finite étale, let $g: Z \to Y$ be the inclusion of the complement of the image of $f$. Then the distributivity law implies that $f_\otimes(0) \wequi g_\sharp(\1)$, which is zero as soon as $Z = \emptyset$.
\end{proof}

\begin{corollary} \label{corr:D-mot-reduced}
Let $\scr X \in \PSh(\Sm_S)_{/B_\et \Sigma_n}$, $n \ge 1$. Then $D^\mot_\scr{X}: \scr C(S) \to \scr C(S)$ is reduced.
\end{corollary}
\begin{proof}
Any colimit of functors preserving the $0$-object preserves the $0$-object, so it suffices to show that $f_\sharp g_\otimes p^*$ preserves the $0$-object, where $g$ is finite étale of degree $n$. Since $f_\sharp$ and $p^*$ are exact they preserve the $0$-object, and for $g_\otimes$ this follows from Lemma \ref{lemm:f-otimes-reduced}, using that $n \ge 1$.
\end{proof}

Finally we can compute the cross effect, in the case $n=2$.
We thank Markus Spitzweck for pointing out the following elegant proof.
\begin{proposition} \label{prop:D2-mot-cr2}
Let $D_2^\mot: \scr C(S) \to \scr C(S)$ denote the motivic extended square. Then $D_2^\mot(E \vee F) \wequi D_2^\mot(E) \vee D_2^\mot(F) \vee E \wedge F$.
\end{proposition}
\begin{proof}
The free normed spectrum functor \[ \NSym: \scr C(S) \to \NAlg(\scr C(S)) \to \scr C(S) \] satisfies $\NSym(E \vee F) \wequi F(E) \wedge F(F)$, and also \cite[Theorem 3.10]{bachmann-MGM} \[ \NSym(E) \wequi \bigvee_{n \ge 0} D_n^\mot(E). \]
The claim can be read off from this.
\end{proof}

\subsubsection{Thom isomorphisms}
If $A \in \NAlg(\SH(S))$, then there is a functor 
\[
\Mod_A(\SH)^\otimes: \Span(\Sm_S, \all, \fet) \to \widehat{\Cat}_\infty,
\] promoting the symmetric monoidal $\infty$-category $\Mod_A(\SH(S))$ to a normed $\infty$-category \cite[Proposition 7.6.4]{norms}.
This comes equipped with a transformation $\SH^\otimes \Rightarrow \Mod_A(\SH)^\otimes$, given in components by $- \otimes A$. We denote the motivic extended powers for the functor $\scr C^\otimes = \Mod_A(\SH)^\otimes$ by 
\[
D^A_n: \Mod_A(\SH(S)) \to \Mod_A(\SH(S)).
\]

\begin{proposition} \label{prop:Dn-thom-iso}
Let $A \in \NAlg(\SH(S))_{\MGL/}$ and $E \in \Mod_A(\SH(S))$. Then there is a canonical equivalence $D^A_n(\Sigma^{2,1} E) \wequi \Sigma^{2n,n} D^A_n(E)$.
\end{proposition}
\begin{proof}
Put $T=\Sigma^{2,1} \1$.
The claim follows from the following computation, carried out in the stable $\infty$-category $\Mod_A(\SH(S))$:
\begin{eqnarray*}
D^A_n(E \otimes T)  & \stackrel{E.\ref{ex:D-gen-colimit}}{\simeq} & \colim_{X \stackrel{f}{\rightarrow} S, U \stackrel{p}{\rightarrow} X} f_{\sharp}p_{\otimes}p^*(E \otimes T)\\
& \stackrel{(2)}{\simeq} & \colim f_{\sharp}( p_{\otimes}p^*E \otimes p_{\otimes}p^*T)\\
& \stackrel{(3)}{\simeq} & \colim f_{\sharp}(p_{\otimes}p^*E \otimes \Sigma^{p_*\scr O_U}\1)\\
& \stackrel{(4)}{\simeq} & \colim f_{\sharp}(p_{\otimes}p^*E \otimes \Sigma^{2n,n}\1)\\
& \stackrel{(5)}{\simeq} & \colim f_{\sharp}(p_{\otimes}p^*E) \otimes \Sigma^{2n,n}\1\\
& \simeq & T^{\otimes n}\otimes \colim f_{\sharp}(p_{\otimes}p^*E)\\
& \stackrel{E.\ref{ex:D-gen-colimit}}{\simeq} & T^{\otimes n} \otimes D^A_n(E).
\end{eqnarray*}
Here, the second equivalence follows from symmetric monoidality of $p_{\otimes}$, the third equivalence follows from the value of $p_{\otimes}$ on spheres \cite[Lemma 4.4]{norms}, the fourth equivalence follows from the coherent orientation on $A$ and the fifth follows from the projection formula. 
\end{proof}

\subsubsection{Extended powers and transfers}
\begin{lemma} \label{lemm:BG-fet}
Let $H \subset G$ be an inclusion of finite groups, with $n = |G:H|$. Then $B_\et H \to B_\et G$ is a relative finite étale morphism of degree $n$.
\end{lemma}
\begin{proof}
A morphism of sheaves being finite étale is fpqc local on the target \cite[Tags 0245 and 02LA, 02VN]{stacks}.
It hence suffices to check that the pullback of $B_\et H \to B_\et G$ along $* \to B_\et G$ (classifying the trivial torsor) is finite étale of degree $n$; this is clear since the pullback is $G/H$.
\end{proof}
\begin{proposition} \label{prop:Dmot-transfer}
Let $H \subset G$ be an inclusion of finite groups, with $d = |G:H|$. Suppose given an action of $G$ on $\{1, \dots, n\}$ and $E \in \SH(S)$.
Then there is a canonical transfer $D^\mot_G(E) \to D^\mot_H(E)$.

Let $Z \in \SH(S)$. If for each $X \in \Sm_S$, $p: Y \to X$ finite étale of degree $d$ and $F \in \SH(X)$ the composite $Z \wedge F \xrightarrow{\tr_p} Z \wedge p_\sharp p^* F \to Z \wedge F$ is an equivalence, then so is the composite $Z \wedge D^\mot_G(E) \to Z \wedge D^\mot_H(E) \to Z \wedge D^\mot_G(E)$.
\end{proposition}
\begin{proof}
This is an immediate consequence of the definition of extended powers as motivic colimits, Lemma \ref{lemm:BG-fet} and Corollary \ref{cor:mot-colimits-transfers}.
\end{proof}

\begin{example}
If $Z = \H\Z/p$ and $d$ is coprime to $p$, the second part of the above proposition applies.
\end{example}

\section{The equivariant model} \label{sec:equivariant-model}
To proceed further in our study of power operations, we will use equivariant motivic homotopy theory as a computational tool. In this section, we introduce an enhancement of the $n$-fold smash powers construction which lands in genuine $\Sigma_n$-motivic spaces/spectra. We will then relate this to the extended power construction. In Appendix~\ref{app:equiv-mot}, we review the necessary technology from equivariant motivic homotopy theory needed to perform the constructions below.

\subsection{Enhanced smash powers} \label{sec:enhanced-smash-powers}
The \emph{enhanced smash power functor} is of the form
\begin{equation}
(-)^{\wedge \underline{n}}: \SH(S) \rightarrow \SH^{\Sigma_n}(S).
\end{equation} 
The composite with the forgetful functor agrees with the $n$-fold smash product functor:
\begin{equation}  \label{eq:stab-smash}
\begin{tikzcd}
\SH(S) \ar{r}{(-)^{\wedge \underline{n}}}\ar{dr}[swap]{(-)^{\wedge n}}  &  \SH^{\Sigma_n}(S) \ar{d}{{\rm forget}}\\
&   \SH(S).
\end{tikzcd}
\end{equation}
From this and the motivic homotopy orbits construction we obtain a functor \[ X \mapsto \EE \Sigma_{n +}\wedge _{\Sigma_n} X^{\wedge \ul{n}} = X^{\wedge \ul n}_{\hh \Sigma_n}, \] which we'll see agrees with the motivic extended powers; see the discussions in Section~\ref{sect:mot-v-enh}. 

The origin of the enhanced smash power functor is the \emph{unstable enhanced smash power}  
\begin{equation}
(-)^{\wedge \underline{n}}: \Spc(S)_* \rightarrow \Spc^{\Sigma_n}(S)_*,
\end{equation} 
which is an enhancement of the endofunctor $X \mapsto X^{\wedge n}$ or the functor given by $p_{\otimes}p^*$ where $p$ is the fold map $p: \coprod_n S \rightarrow S$:
\begin{equation} \label{eq:unstab-smash}
\begin{tikzcd}
\Spc(S)_* \ar{r}{(-)^{\wedge \underline{n}}}\ar{dr}[swap]{p_{\otimes}p^*}  &  \Spc^{\Sigma_n}(S)_* \ar{d}{{\rm forget}}\\
&   \Spc(S)_*.
\end{tikzcd}
\end{equation}

As hinted in the diagrams ~\eqref{eq:stab-smash} and~\eqref{eq:unstab-smash}, there is an intimate relationship between the smash power functors and the norm functors $p_{\otimes}$ of the first author and Hoyois constructed in \cite[Section 3]{norms}.
In fact, the enhanced smash powers are simply norms along certain finite étale morphisms of stacks.

Recall that if $p: \scr X \to \scr Y$ is a finite étale morphism of stacks, then there are symmetric monoidal norm functors \[ p_\otimes: \Spc(\scr X) \to \Spc(\scr Y), \] and similarly for $\Spc(\ph)_*$ and $\SH(\ph)$, provided that the latter is defined \cite[Section 3.4]{bachmann-MGM}.
By definition, for $X \in \Sm_{\scr X}$ we have $p_\otimes(X) = R_p(X)$, where $R_p$ denotes Weil restriction.

Now let $S$ be a scheme and consider the $\Sigma_n$-scheme $\ul{n} = S^{\coprod n}$ over $S$.
Then $p: [\ul{n}/ \Sigma_n] \to [S / \Sigma_n]$ is a finite étale morphism.
\begin{definition}
We put \[ (\ph)^{\wedge \ul n} = p_\otimes p^*\triv: \Spc(S)_* \to \Spc^{\Sigma_n}(S)_*, \] and similarly for $\SH(\ph), \Spc(\ph)$.
\end{definition}
In particular on the level of schemes, $X^{\times \ul n} = X^n$ with the $\Sigma_n$-action given by permuting the factors.

\subsection{Diagonals}
While we do not need it in this article, we record for future reference that enhanced smash powers of pointed spaces admit diagonals, which we construct as follows.
We have the functors \[ \triv, (\ph)^{\wedge \ul n}: \Sm_{S+} \to \Sm_{S+}^{\Sigma_n}, \quad X_+ \mapsto X_+, (X^{\times \ul{n}})_+ \] which on applying $\PSh_\Sigma$ and using \cite[Lemma 2.1]{norms} induce \[ \triv, (\ph)^{\wedge \ul n}: \PSh_\Sigma(\Sm_S)_* \to \PSh_\Sigma(\Sm_{S+}^{\Sigma_n})_*. \]
The maps $\Delta_n: X_+ \mapsto (X^{\times n})_+$ of schemes are $\Sigma_n$-invariant if $X$ is given the trivial action and $X^n$ the permutation action.
They are also natural in the pointed scheme $X_+$.
Hence they induce \[ \Delta_n: \triv \Rightarrow (\ph)^{\wedge \ul n}: \Sm_{S+} \to \Sm_{S+}^{\Sigma_n}. \]
\begin{definition}\label{def:diagonal-transform}
We denote by \[ \Delta_n: \triv \Rightarrow (\ph)^{\wedge \ul n}: \Spc(S)_* \to \Spc^{\Sigma_n}(S)_* \] the induced transformation.
\end{definition}

\subsection{Motivic extended powers via enhanced smash powers} \label{sect:mot-v-enh} We will now construct the motivic extended powers via enhanced smash powers. For a group $G$, we have $\EE G \in \Spc^G(S)[\scr F_\triv]$ as in Definition~\ref{def:eg}.
Recall the motivic homotopy orbits functor \[ (\ph)_{\hh G}: \SH^G(S) \xrightarrow{\wedge \EE G_+} \SH^G(S)[\scr F_\triv] \xrightarrow{(\ph)/G} \SH^G(S) \] from Section \ref{subsec:quotients}.

\begin{definition} \label{def:orbits-enh}
We denote by $\DD_n$ the functor \[ \DD_n: \SH(S) \xrightarrow{(\ph)^{\wedge \ul{n}}} \SH^{\Sigma_n}(S) \xrightarrow{(\ph)_{\hh \Sigma_n}} \SH(S). \]
If $H \subset \Sigma_n$ is a subgroup, then we denote by $\DD_H$ the functor \[ \DD_H: \SH(S) \xrightarrow{(\ph)^{\wedge \ul{n}}} \SH^{\Sigma_n}(S) \to  \SH^{H}(S) \xrightarrow{(\ph)_{\hh H}} \SH(S). \]
\end{definition}

Our next goal is to prove a canonical equivalence of endofunctors  $\D_n^\mot \simeq \DD_n$. To do so, recall that we have the presheaf $B_\et \Sigma_n$ of principal $\Sigma_n$-bundles. Given a principal $\Sigma_n$-bundle $T \to X$, we can form the quotient $T \times_{\Sigma_n} \{1, \dots, n\}$ (in the étale topology); this is a finite étale scheme over $X$ of rank $n$. In this way we obtain a morphism of presheaves $B_\et \Sigma_n \to \FEt^{\wequi,n}$ elaborating the equivalence already stated in Lemma~\ref{lem:stk}.

\begin{proposition}
Let $p: T \to S$ be a finite étale morphism of degree $n$, and $q: R \to S$ the associated $\Sigma_n$-torsor.
Then the we have an equivalence \[ p_\otimes p^* \wequi R_+ \wedge_{\Sigma_n} (\ph)^{\wedge \ul{n}}: \SH(S) \to \SH(S), \] naturally in $p$.
\end{proposition}
\begin{proof}
The right hand functor can alternatively be written as \[ \SH(S) \xrightarrow{(\ph)^{\wedge \ul n}} \SH([S/\Sigma_n]) \xrightarrow{[q/\Sigma_n]^*} \SH([R/\Sigma_n]) \wequi \SH(S). \]
In particular it is symmetric monoidal and preserves sifted colimits.
By the universal property of $\SH(S)$ \cite[Lemma 4.1]{norms}, it suffices to exhibit a symmetric monoidal equivalence of the two functors after composing with $\SmQP_{S+} \to \SH(S)$.
This reduces to exhibiting for $X_+ \in \SmQP_{S+}$ a natural equivalence \[ X^T_+ \wequi (R \times_{\Sigma_n} X^n)_+. \]
Both sides being étale sheaves, we may assume that $T = S^{\coprod n}$ whence $R = S \times \Sigma_n$, in which case the claim is clear.
\end{proof}

\begin{remark} \label{rmk:fundamental-diagrams-comparison}
It follows that under the equivalence $B_\et \Sigma_n \wequi \FEt^{\wequi,n}$ the functors
\[ (\Sm_S)_{\sslash \FEt^{\wequi,n}} \times \SH(S) \to \SH(S), ((T \xrightarrow{p} X), E) \mapsto (X \to S)_\sharp p_\otimes E_T \]
can be equivalently described as 
\[ (\Sm_S)_{\sslash B_\et \Sigma_n} \times \SH(S) \to \SH(S), ((R \to X), E) \mapsto (X \to S)_\sharp R_+ \wedge_{\Sigma_n} E^{\wedge \ul{n}}_X. \]
\end{remark}

\begin{lemma} \label{lemm:EE-model}
There is a canonical equivalence $\EE G \wequi \colim_{R \in B_\et G} R \in \Shv_{\Nis}(\Sm^G_S)$.
\end{lemma}
\begin{proof}
For this proof we put $\Sm_S^G = \Sm_{S \sslash G}^\qproj$ and $\Sm_S = \Sm_S^\qproj$ (see the discussion in Section \ref{subsec:stacks-G-schemes}).
By this convention, whenever $X \in \Sm_S^G$ carries a free $G$-action, then the quotient exists and lies in $\Sm_S$.
This implies that the category of elements $(\Sm_S)_{\sslash B_\et G}$ is equivalent to $\Sm_S^G[\scr F_\triv]$.
It follows that \[ \EE G \wequi \colim_{R \to \EE G} R \wequi \colim_{R \in \Sm_S^G[\scr F_\triv]} R \wequi \colim_{R \in B_\et G} R, \] where the colimits are computed in $\PSh(\Sm_S^G)$.
The desired result about sheaves follows immediately.\footnote{The only reason this lemma is stated in terms of sheaves is that otherwise we are not free to choose the definition of $\Sm_S^G$.}
\end{proof}

\begin{corollary} \label{cor:d-vs-d}
The two functors $\D_H^\mot, \DD_H: \SH(S) \to \SH(S)$ are canonically equivalent.
\end{corollary}
\begin{proof}
By definition, $\D_H^\mot(\ph)$ is given by the colimit over $B_\et H \to B_\et \Sigma_n \wequi \FEt^{\wequi, n}$ of the first diagram of Remark \ref{rmk:fundamental-diagrams-comparison}. It thus remains to show that the colimit over $B_\et H$ of the second diagram is $\DD_H(\ph)$.
This is true since \[ \colim_{R \in B_\et H} \left[ (R \times_H \Sigma_n)_+ \wedge (\ph)^{\wedge \ul n} \right] \wequi (\colim_{R \in B_\et H} R)_+ \wedge_H (\ph)^{\wedge \ul n} \wequi \EE H_+ \wedge_H (\ph)^{\wedge \ul n} \wequi \DD_H(\ph), \] where in the middle we have used Lemma \ref{lemm:EE-model} (and we have also used that $(\ph)/H$ and $(\ph) \wedge E$ preserve colimits).
\end{proof}

\subsection{Monoidality of motivic extended powers}
We illustrate the convenience of the equivariant model by establishing further properties of $D_n^\mot$.
While it would be possible to do so in the motivic colimit model, this seems to require much more effort.

\begin{proposition} \label{prop:Dmot-oplax-monoidal}
The functor $D_n^\mot: \SH(S) \to \SH(S)$ is canonically op-lax symmetric monoidal.
\end{proposition}
\begin{proof}
Since $D_n^\mot \wequi \DD_n$ by Corollary \ref{cor:d-vs-d}, and the functor $\DD_n$ is a composite of oplax symmetric monoidal functors (see Proposition \ref{prop:quot}(3)), the result follows.
\end{proof}

\begin{definition} \label{def:op-lax-witness}
Let $A, B \in \scr C(S)$. We denote by $\beta_{A, B} = \beta_{A,B}^n: D_n^\mot(A \wedge B) \to D_n^\mot(A) \wedge D_n^\mot(B)$ the natural transformation witnessing op-lax monoidality of $D_n^\mot$, as established in Proposition \ref{prop:Dmot-oplax-monoidal}.

Consider the map $\alpha: * \to \FEt^{\wequi,n}$ corresponding to $S^{\coprod n} \in \FEt^{\wequi,n}_S$. This induces a natural transformation $\alpha_A = \alpha_A^n: A^{\wedge n} \wequi D^\mot_\alpha(A) \to D_n^\mot(A)$.
\end{definition}

\begin{remark} \label{rmk:alpha-naturality}
By construction, $\alpha_A$ is natural in $A$: if $f: A \to B \in \scr C(S)$, then the following square commutes
\begin{equation*}
\begin{CD}
A^{\wedge n} @>{\alpha_A}>> D_n^\mot(A) \\
@V{f^{\wedge n}}VV          @V{D_n^\mot(f)}VV \\
B^{\wedge n} @>{\alpha_A}>> D_n^\mot(B). \\
\end{CD}
\end{equation*}
\end{remark}

The following result will eventually be used to establish the Cartan formula for motivic power operations.

\begin{theorem}[Ur-Cartan relation.] \label{thm:ur-cartan-reln}
Let $E \in \SH(S)$. The composite \[ D_2^\mot(E \wedge E) \xrightarrow{\beta_{E,E}} D_2^\mot(E) \wedge D_2^\mot(E) \xrightarrow{\alpha_{D_2^\mot(E)}} D_2^\mot(D_2^\mot(E)) \] is homotopic to $D_2^\mot(\alpha_E)$.
\end{theorem}
\begin{proof}
Denote by $W \subset \Sigma_4$ the wreath product $\Sigma_2 \wr \Sigma_2$.
$W$ has a unique normal subgroup $H$ of order $2$, and a unique normal non-cyclic subgroup $H \subset K \subset W$.
($H$ is generated by $(13)(24)$ and $K$ by $(13),(24)$.)
We claim that (1) $D_2^\mot(D_2^\mot(E)) \wequi E^{\wedge \ul 4}_{\hh W}$ (see also Proposition \ref{prop:Dmot-iterated-wreath}).
We claim that similarly (2) $D_2^\mot(E \wedge E) \wequi E^{\wedge \ul 4}_{\hh H}$ and (3) $D_2^\mot(E) \wedge D_2^\mot(E) \wequi E^{\wedge \ul 4}_{\hh K}$.
We finally claim that (4) commutative diagram in $\SH^{\Sigma_4}(S)$
\begin{equation*}
\begin{tikzcd}
E^{\wedge \ul 4}\wedge \Sigma_4/H_+ \ar[r] \ar[d] & E^{\wedge \ul 4} \wedge \Sigma_4/W_+ \\
E^{\wedge \ul 4} \wedge \Sigma_4/K_+ \ar[ur]
\end{tikzcd}
\end{equation*}
turns upon applying $(\ph)_{\hh\Sigma_4}$ into the commutative diagram
\begin{equation*}
\begin{tikzcd}
D_2^\mot(E \wedge E) \ar[r, "D_2^\mot(\alpha_E)"] \ar[d, "\beta_{E,E}"] & D_2^\mot(D_2^\mot(E)) \\
D_2^\mot(E) \wedge D_2^\mot(E) \ar[ur, "\alpha_{D_2^\mot(E)}" swap]
\end{tikzcd}
\end{equation*}
that we were after.

To prove the claims\tombubble{This seems unsatisfyingly messy}, note that the composite \[ F: \SH^{\Sigma_2}(S)[\scr F_\triv] \xrightarrow{(\ph)/\Sigma_2} \SH(S) \xrightarrow{(\ph)^{\wedge \ul 2}} \SH^{\Sigma_2}(S) \] is a $\Spc(S)_*$-module functor, provided we act by $\triv$ on $\SH^{\Sigma_2}(S)[\scr F_\triv]$ and by the norm on $\SH^{\Sigma_2}(S)$.
It is consequently obtained via universal properties from \[ \Sm_{S+}^{\Sigma_2}[\scr F_\triv] \to \Sm_{S+}^{\Sigma_2}, X \mapsto (X/\Sigma_2)^{\times \ul 2}_+. \]
Now observe that $(X/\Sigma_2)^{\times \ul 2} \wequi X^{\times \ul 2}/K$.
Tracing through the universal properties again, we see that $F$ is homotopic to the composite \[ F': \SH^{\Sigma_2}(S)[\scr F_\triv] \xrightarrow{(\ph)^{\wedge \ul 2}} \SH^{W}(S)[\scr F_{K\triv}] \xrightarrow{(\ph)/K} \SH^{W/K}(S) \wequi \SH^{\Sigma_2}(S). \]
From $F \wequi F'$ we deduce that \[ D_2^\mot(E)^{\wedge \ul 2} \wequi E^{\wedge \ul 4}_{\hh K}, \] with the inherited action by $\Sigma_2 \wequi W/K$.
This implies claims (1) and (3).
Claim (2) is proved similarly, and claim (4) now follows directly from the definitions.
\end{proof}

\begin{subappendices}
\section{Small and large presheaves} \label{app:small-and-large-presheaves}
In this appendix, by \emph{category} we shall always mean \emph{$\infty$-category}.
\subsection{Presheaves on large categories} \label{sec:small-large-presheaves} \todo{is there a reference?}
For universes $\bb U \subset \bb V$ we call a $(\bb U, \bb V)$-category a category $\scr C$ with $Ob(\scr C) \in \bb V$ and $\Map_{\scr C}(c, d) \in \bb U$ for all $c, d \in \scr C$. We fix universes $\bb U \in \bb V \in \bb W$. We call a $(\bb U, \bb U)$-category a \emph{small} category and a $(\bb U, \bb V)$-category a \emph{locally small} category. By a large category we mean any category that is not locally small, for example a $(\bb V, \bb V)$-category. Note that $\Cat$ (the category of small categories) is a locally small category, and $\widehat{\Cat}$ (the category of locally small categories) is a large category, more precisely a $(\bb V, \bb W)$-category.

If $\scr C$ is a locally small category, there are various possible categories of presheaves on $\scr C$.  First, we denote by $\widehat{\Spc}$ the large category (more precisely a $(\bb V, \bb W)$-category) of $\bb V$-small spaces. We define:

\begin{enumerate}
\item The $\infty$-category of \emph{presheaves of spaces} as $\PSh_1(\scr C) = \Fun(\scr C^\op, \Spc).$
\item The $\infty$-category of \emph{presheaves valued in $\widehat{\Spc}$, i.e., those that take values in $\bb V$-small spaces} as $\widehat{\PSh}(\scr C) = \Fun(\scr C^\op, \widehat{\Spc})$.
\item In $\PSh_1(\scr C)$ we denote by $\PSh_0(\scr C)$ the full subcategory generated under $\bb U$-small colimits by $\scr C$.
\item In $\widehat{\PSh}(\scr C)$ we denote by $\widehat{\PSh}_0(\scr C)$ the full subcategory generated under $\bb U$-small colimits by $\scr C$.
\end{enumerate}\begin{lemma} \label{lemm:P-large}
Let $\scr C$ be a locally small category.
\begin{enumerate}
\item The canonical functor $\PSh_1(\scr C) \to \widehat{\PSh}(\scr C)$ is fully faithful and preserves $\bb U$-small colimits. The restricted functor $\PSh_0(\scr C) \to \widehat{\PSh}_0(\scr C)$ is an equivalence.
\item Let $\scr D$ be any (possibly large) $\infty$-category with $\bb U$-small colimits. Let $\Fun_{\bb U}(\widehat{\PSh}_0(\scr C), \scr D) \subset \Fun(\widehat{\PSh}_0(\scr C), \scr D)$ denote the full subcategory on those functors preserving $\bb U$-small colimits. Then composition with the Yoneda embedding induces an equivalence $\Fun_{\bb U}(\widehat{\PSh}_0(\scr C), \scr D) \to \Fun(\scr C, \scr D)$.
\item Every object in $\PSh_0(\scr C)$ can be written as a $\bb U$-small colimit of representable presheaves.
\item The category $\PSh_0(\scr C)$ is locally small.
\end{enumerate}
\end{lemma}
\begin{proof}
(1) The first claim follows from the fact that $\Spc \to \widehat{\Spc}$ is fully faithful and preserves $\bb U$-small colimits.\NB{E.g. because if $X \to Y \in \widehat\Spc$ with $Y$ and all fibers $\bb U$-small, then $X$ is also $\bb U$-small.} The second claim is an immediate consequence of the definitions.

(2) This follows from \cite[Proposition 5.3.6.2 and its proof]{HTT}.

(3) Let $\scr P \subset \PSh_0(\scr C)$ be the full subcategory on $\bb U$-small colimits of representables. It suffices to show that $\scr P$ is closed under $\bb U$-small colimits. Let $F: I \to \scr P$ be a $\bb U$-small diagram. For each $i \in I$ choose a $\bb U$-small diagram $G_i: J_i \to \scr C \subset \PSh_0(\scr C)$ with $\colim G_i \wequi F(i)$. Let $\scr C'$ be the full subcategory of $\scr C$ on $\bigcup_i G_i Ob(J_i)$; this is a $\bb U$-small category. The functor $\scr C' \to \scr C$ induces a colimit preserving functor $e: \PSh_1(\scr C') \to \PSh_1(\scr C)$, which is fully faithful by (1). By construction $F$ factors through the essential image $\scr P'$ of $e$ and hence $\colim_I F \in \scr P'$. It remains to show that $\scr P' \subset \scr P$. But every object in $\PSh_1(\scr C') = \PSh(\scr C')$ is a $\bb U$-small colimit of objects in the Yoneda image of $\scr C'$, so this is clear.

(4) $\PSh_1(\scr C)$ is a $(\bb V, \bb V)$-category. It hence suffices to show that the mapping spaces in $\PSh_0(\scr C)$ are $\bb U$-small. This follows from (3).
\end{proof}

\begin{definition} \label{def:P-large-convention}
If $\scr C$ is a locally small category, we denote by $\PSh(\scr C)$ the equivalent locally small categories $\PSh_0(\scr C) \wequi \widehat{\PSh}_0(\scr C)$.
\end{definition}

We can strengthen Lemma \ref{lemm:P-large}(2) somewhat. Let $G: \scr D \to \scr E$ be a functor and $i: \scr C \hookrightarrow \scr E$ a full subcategory. We call a functor $F: \scr C \to \scr D$ a \emph{partial left adjoint to $G$} if there exists a transformation $u: i \to GF$ such that for all $c \in \scr C$ and $d \in \scr D$ the induced map $\Map_\scr{D}(Fc, d) \to \Map_\scr{E}(GFc, Gd) \to \Map_\scr{E}(c, Gd)$ is an equivalence.

\begin{lemma} \label{lemm:big-P-partial-adj}
Let $\scr C$ be locally small and $\scr D$ an $\infty$-category with $\bb U$-small colimits and $\bb V$-small mapping spaces. Suppose $f: \scr C \to \scr D$ is any functor. Then the functor $f^*: \scr D \to \widehat{\PSh}(\scr C)$ has as a partial left adjoint given by the $\bb U$-cocontinuous extension $f: \PSh(\scr C) \to \scr D$.
\end{lemma}
\begin{proof}
Let $j: \scr D \to \scr D'$ be a functor such that (1) $j$ is fully faithful, (2) $j$ preserves $\bb U$-small colimits, and (3) $\scr D'$ has $\bb V$-small colimits. This exists by \cite[Proposition 5.3.6.2]{HTT}. Applying the universal property of $\scr C \to \widehat{\PSh}(\scr C)$ to the composite $jf$, we obtain a $\bb V$-cocontinuous extension $\bar{f}: \widehat{\PSh}(\scr C) \to \scr D'$. The composite $\PSh(\scr C) \to \widehat{\PSh}(\scr C) \xrightarrow{\bar f} \scr D'$ preserves $\bb U$-small colimits and hence, by the universal property of $\PSh(\scr C)$, is determined by its restriction along $\scr C \to \PSh(\scr C)$. Thus the following diagram commutes
\begin{equation*}
\begin{CD}
\PSh(\scr C) @>f>> \scr D  \\
@ViVV               @VjVV  \\
\widehat{\PSh}(\scr C) @>\bar{f}>> \scr D'. \\
\end{CD}
\end{equation*}
Fully faithfulness of $j$ now implies that $f^* \wequi \bar f^* j$. The unit transform $\bar u$ of the adjunction $\bar f \dashv \bar f^*$ provides $u: i \Rightarrow \bar f^* \bar f i \wequi \bar f^* j f \wequi f^* f$. Fully faithfulness of $i$ and $j$ finally imply that $u$ exhibits $f$ as partial left adjoint to $f^*$.
\end{proof}

\subsection{Presheaves and slices} \todo{is there a reference?}
Recall that if $F: \scr C \to \scr D$ is a functor (with $\scr D$ possibly large) and $d \in \scr D$, then $\scr C_{/d} \to \scr C$ is a cartesian fibration classifying the functor $\scr C^\op \to \widehat{\Spc}, a \mapsto \Map(F(\ph), b)$. Note that this category could a priori be as large as $\scr D$.

Let $\scr C$ be a small category and denote by $F: \scr C \to \Fun(\scr C^\op, \widehat{\Cat}_\infty)$ the composition of the Yoneda embedding with the inclusion $\Spc \to \widehat{\Cat}_\infty$. Suppose given a functor $\scr F: \scr C^\op \to \widehat{\Cat}_\infty$. We denote by $\scr C_{/\scr F}$ the slice category along $F$ after identifying $\scr C$ with its Yoneda image. Note that objects of $\scr C_{/\scr F}$ can be informally described as pairs $(c, X)$ with $c \in \scr C$ and $X: Fc \to \scr F$, i.e. $X \in \scr F(c)$.

\begin{lemma} \label{lemm:presheaf-slice-cat}
\begin{enumerate}
\item $\scr C_{/\scr F} \to \scr C$ is the cartesian fibration classifying the functor $\scr C \to \widehat{\Spc}$, $c \mapsto \scr F(c)^\wequi$.
\item The large category $\scr C_{/\scr F}$ is locally small.
\end{enumerate}
\end{lemma}
\begin{proof}
(1) As recalled above, $\scr C_{/\scr F} \to \scr C$ is the cartesian fibration classifying the functor $c \mapsto \Map_{\Fun(\scr C^\op, \widehat{\Cat}_\infty)}(Fc, \scr F)$.  We have equivalences
\begin{align*}
 \Map_{\Fun(\scr C^\op, \widehat{\Cat}_\infty)}(Fc, \scr F) &\wequi \Map_{\Fun(\scr C^\op, \widehat{\Spc}_\infty)}(Fc,  \scr F^\wequi) \\
                                                            &\wequi  \scr F(c)^\wequi,
\end{align*}
where the first equivalence is because taking maximal subgroupoid is right adjoint to the inclusion $\Spc \to \Cat_\infty$ and the second equivalence is the Yoneda lemma.

(2) In other words, mapping spaces in $\scr C_{/\scr F}$ are small.
Let $(c, X), (d, Y) \in \scr C_{/\scr F}$.
There exists a fully faithful subfunctor $\scr F' \subset \scr F$ valued in small categories containing $X$ and $Y$.
Since $\Map_{\scr C_{\scr F}}((c,X),(d,Y)) \wequi \Map_{\scr C_{\scr F'}}((c,X),(d,Y))$ (e.g. by (1)), we have reduced to $\scr F$ small, and the result is clear.
\end{proof}

Now consider the categories $\PSh(\scr C_{/\scr D})$ (see Definition \ref{def:P-large-convention}) and $\widehat{\PSh}(\scr C)$. Note that $\scr D^\wequi$ naturally defines an object of $\widehat{\PSh}(\scr C)$, and that $\PSh(\scr C) \subset \widehat{\PSh}(\scr C)$ is a full subcategory. We denote by $\PSh(\scr C)_{/\scr D}$ the full subcategory of $\widehat{\PSh}(\scr C)_{/\scr D^\wequi}$ (ordinary slice category) on objects $(P, E)$ with $P \in \PSh(\scr C)$. Since $\widehat{\PSh}(\scr C)$ has colimits so does $\widehat{\PSh}(\scr C)_{/\scr D^\wequi}$; consequently by Lemma \ref{lemm:P-large} the functor $\scr C_{/\scr D} \to \widehat{\PSh}(\scr C)_{/\scr D^\wequi}$ induces a functor $\PSh(\scr C_{/\scr D}) \to \widehat{\PSh}(\scr C)_{/\scr D^\wequi}$.

\begin{lemma} \label{lemm:psh-slice-comp}
The functor $\PSh(\scr C_{/\scr D}) \to \widehat{\PSh}(\scr C)_{/\scr D^\wequi}$ factors through $\PSh(\scr C)_{/\scr D} \subset \widehat{\PSh}(\scr C)_{/\scr D^\wequi}$ and induces an equivalence $\PSh(\scr C_{/\scr D}) \wequi \PSh(\scr C)_{/\scr D}$.
\end{lemma}
\begin{proof}
The functor $\PSh(\scr C_{/\scr D}) \to \widehat{\PSh}(\scr C)_{/\scr D^\wequi}$ preserves colimits. Since colimits in section categories are computed on the source \cite[Corollary 5.1.2.3]{HTT}, the functor factors as claimed. It suffices to show that $\PSh(\scr C)_{/\scr D}$ is generated under small colimits by the completely compact objects $\scr C_{/\scr D} \subset \PSh(\scr C)_{/\scr D}$. Let $c \in \scr C$, $F \in \PSh(\scr C)$ and $\alpha: c \to \scr D, \beta: F \to \scr D$. We have $\Map_{\PSh(\scr C)_{/\scr D}}(\alpha, \beta) = F(c) \times_{\scr D(c)^\wequi} \{\alpha\}$. Since colimits in section categories are computed on the source and colimits in $\Spc$ are universal, this expression preserves colimits in $F$. It follows that the objects $\scr C_{/\scr D} \subset \PSh(\scr C)_{/\scr D}$ are completely compact. Given $\beta: F \to \scr D \in \PSh(\scr C)_{/\scr D}$, we have $\beta \wequi \colim_{\gamma \in \scr C_{/F}} \beta \circ \gamma$ (again since colimits are computed on the source), so the objects $\scr C_{/\scr D} \subset \PSh(\scr C)_{/\scr D}$ generate. This concludes the proof.
\end{proof}

\section{Some stacks left Kan extended from smooth affines}
\label{subapp:kan-extension}

In this appendix, we slightly extend the growing literature on left Kan extensions of invariants from smooth schemes. Let $S$ be an affine scheme and $\scr X \in \PSh(\Aff_S)$.
We call $\scr X$ \emph{left Kan extended from smooth affines} if the counit map $e_!e^* \scr X \to \scr X$ is an equivalence, where \[ e_!: \PSh(\SmAff_S) \adj \PSh(\Aff_S): e^*. \]

\begin{lemma} \label{lemm:LKE-primitive}
Let $G \to S$ be a finite étale group scheme, $X$ a finite étale $G$-scheme, and $Y$ a smooth $S$-scheme\NB{Or smooth stack with quasi-affine diagonal...}.
Then the Weil restriction $Q(G, X, Y)$ of $Y \times [X/G]$ along $[X/G] \to [*/G]$ is left Kan extended from smooth affines.
\end{lemma}
\begin{proof}
Put $\scr X := Q(G, X, Y)$.
By \cite[Proposition A.0.4]{EHKSY3} it suffices to show that $\scr X$ is a smooth algebraic stack with quasi-affine diagonal.
For this it suffices to prove that (1) $\scr X \to [*/G]$ is smooth with quasi-affine diagonal and (2) $[*/G] \to *$ is smooth with affine diagonal\NB{diagonal of $X \to Y \to Z$ factors as $X \to X \times_Y X \to X \times_Z X$ and the latter map is pullback of $Y \to Y \times_Z Y$}.
Since $[X/G] \to [*/G]$ is finite étale, the Weil restriction $\scr X$ is étale locally on $[*/G]$ given by $[*/G] \times Y^n$, which is smooth and schematic, hence has quasi-affine diagonal; (1) follows since it is fppf local on $[*/G]$.
(2) is clear.\NB{The map $* \to [*/G]$ is a covering by schemes, so the stack is algebraic.
The pullback of this covering along itself $G \to *$ which is finite étale, whence $[*/G] \to *$ is smooth (this property being étale local on the source).
Finally the diagonal $[X/G] \to [X/G] \times [X/G]$ is covered by the action map $X \times G \to X \times X$, which is certainly affine if $X=*$, since $G$ is.}
\end{proof}

\begin{example} \label{ex:Bet-LKE}
Taking $Y = X = *$, we see that $Q(G,*,*) = [*/G]$ is left Kan extended from smooth affines.
\end{example}
\begin{example} \label{ex:Q-weil-res}
Taking $G=\Sigma_n$, $X = \ul{n}$ and $Y$ arbitrary, $[X/G] \to [*/G]$ is the universal finite étale scheme of degree $n$ and $Q(\Sigma_n, \ul{n}, Y)$ is the stack of finite étale schemes with a morphism to $Y$ (which is thus left Kan extended from smooth affines if $Y$ is smooth over $S$).
\end{example}

\begin{lemma} \label{lemm:stk-coprod}
Let $\scr X_i$ be a sequence of stacks.
Then \[ L_\Sigma \coprod^{\PSh}_i \scr X_i \wequi \colim_i \coprod_{i \le n}^{\Stk} \scr X_i, \] where on the left hand side we take the coproduct in presheaves and on the right hand side we mean the filtered colimit in presheaves, but the finite coproduct in stacks.
\end{lemma}
\begin{proof}
Since filtered colimits of spaces preserve finite products, the right hand side is a $\Sigma$-presheaf.
It follows that it suffices to show that $L_\Sigma \left( \scr X_1 \amalg^\PSh \scr X_2 \right) \wequi \scr X_1 \amalg^\Stk \scr X_2$.
A morphism $T \to \scr X_1 \amalg \scr X_2$ is the same as a disjoint union decomposition $T \wequi T_1 \amalg T_1$ and morphisms $T_i \to \scr X_i$.
This is exactly a section of the sheafification of $\scr X_1 \amalg^\PSh \scr X_2$ with respect to the topology of disjoint unions, i.e. $L_{\Sigma}$ \cite[Lemma 2.4]{norms}.
\end{proof}

Denote by $Q(\FEt, Y)$ the stack of finite étale schemes together with a morphism to some fixed scheme $Y$.
\begin{corollary} \label{cor:stk-fet}
For $Y$ smooth over $S$ we have \[ L_{\Sigma}\coprod_{n \geq 0} Q(\Sigma_n, \ul{n}, Y) \wequi Q(\FEt, Y). \]
In particular \[ L_\Sigma \coprod_{n \ge 0} B_\et \Sigma_n \wequi \FEt^\wequi. \]
\end{corollary}
\begin{proof}
Immediate from Lemma \ref{lemm:stk-coprod} and Example \ref{ex:Q-weil-res}, using that $B_\et \Sigma_n$ is the stack of finite étale schemes of degree $n$, and that the degree of a finite étale scheme is locally constant.
\end{proof}

\begin{lemma} \NB{ref?} \label{lemm:LKE-sigma}
The left Kan extension functor \[ e_!: \PSh(\SmAff_S) \to \PSh(\Aff_S) \] preserves $\Sigma$-presheaves and $L_\Sigma$-equivalences.
\end{lemma}
\begin{proof}
The claim about $L_\Sigma$-equivalences is clear.
The preservation of $\Sigma$-presheaves follows from the facts that for $X, Y \in \Aff_S$ we have $(\SmAff_S)_{X \amalg Y/} \wequi (\SmAff_S)_{X/} \times (\SmAff_S)_{X/}$, these categories are filtered, and filtered colimits of spaces commute with finite products.
\end{proof}

\begin{corollary} \label{corr:FEt-base-change}
Let $Y \in \Sm_S$.
Then $Q(\FEt, Y) \in \PSh(\Aff_S)$ is left Kan extended from smooth affines.
\end{corollary}
\begin{proof}
Via Lemma \ref{lemm:LKE-sigma} and Corollary \ref{cor:stk-fet} we are reduced to the stacks $Q(\Sigma_n, \ul{n}, Y)$, which are treated in Lemma \ref{lemm:LKE-primitive}.
\end{proof}

We now treat the case where $S$ is not necessarily affine.
\begin{lemma} \label{lemm:LKE-trick}
Let $S$ be a scheme and $\scr X \in \PSh(\Sch_S)$.
Suppose $\scr X|_{\Aff_B}$ is left Kan extended from smooth affines, for all $B$ in some Zariski cover of $S$ by affine schemes.
Denote by \[ \tilde e_!: \PSh(\Sm_S) \adj \PSh(\Sch_S): \tilde{e}^* \] the left Kan extension adjunction.
Then the counit $\tilde e_! \tilde e^* \scr X \to \scr X$ is a Zariski equivalence.
\end{lemma}
\begin{proof}
This is a minor variation of \cite[Lemma 3.3.9]{EHKSY3}, with essentially the same proof.
\end{proof}

\begin{lemma} \label{lemm:LKE-trick2}
Let $f: T \to S$ be a morphism of schemes and $\scr X \in \PSh(\Sch_S)$ such that both $\scr X$ and $\scr X|_{\Sch_T}$ are left Kan extended from their restrictions to smooth schemes up to $L_\Zar$ (e.g. satisfy the assumptions of Lemma \ref{lemm:LKE-trick}).
Then the canonical map \[ \alpha: f^*(\scr X|_{\Sm_S}) \to \scr X|_{\Sm_T} \in \PSh(\Sm_T) \] is a Zariski equivalence.
\end{lemma}
\begin{proof}
Since in the adjunction \[ g_!: \PSh(\Sm_T) \adj \PSh(\Sch_T): g^* \] the functor $g_!$ is fully faithful and both functors preserve Zariski equivalences, it suffices to prove that $g_!(\alpha)$ is an equivalence.
Compatibility of left Kan extensions with composition and our assumptions imply that both $g_! f^*(\scr X|_{\Sm_S}) \wequi f^* g_!(\scr X|_{\Sm_S})$ and $g_!(\scr X|_{\Sm_T})$ are Zariski equivalent to $\scr X|_{\Sch_T}$, whence the result.
\end{proof}

\begin{corollary} \label{corr:Bet-stable-by-base-change}
Let $f: T \to S$ be a morphism of schemes, and $G$ a finite étale $S$-group scheme.
There is a canonical morphism $f^* B_\et G \to B_\et f^* G \in \PSh(\Sm_T)$, which is a Zariski-equivalence.
\end{corollary}
\begin{proof}
Combine Lemmas \ref{lemm:LKE-trick}, \ref{lemm:LKE-trick2} and Example \ref{ex:Bet-LKE}.
\end{proof}

\section{Recollections on equivariant motivic homotopy theory} \label{app:equiv-mot}
In this section, $G$ is always a finite group, thought of as a finite discrete group scheme over $S$.
We review the construction and basic properties of $G$-equivariant motivic homotopy theory.
Our main reference is \cite{bachmann-MGM}, which builds on \cite{gepner-heller}, \cite{khan2021generalized} and \cite{Hoyois:6functor}.

\subsection{Stacks and $G$-schemes} \label{subsec:stacks-G-schemes}
If $S$ is a scheme with a $G$-action, we can form the stacky quotient $S \sslash G = [S/G]$.\footnote{The notation $S \sslash G$ is more common in homotopy theory and used in \cite{bachmann-MGM}, whereas $[S/G]$ is used in algebraic geometry.}
If $B$ is a base scheme with trivial $G$-action, then the functor \[ \Sch_B^G \to \Stk_{/B \sslash G}, X \mapsto X \sslash G \] is fully faithful (where $\Stk$ denotes the $(2,1)$-category of algebraic stacks).
In this way the language of stacks can replace the language of $G$-schemes, and this is indeed done in our main reference.

Let $\scr X$ be a stack.
Denote by \[ \Sm_{\scr X}^{\mathrm{aff}} \subset \Sm_{\scr X}^{\mathrm{qaff}} \subset \Sm_{\scr X}^{\qproj} \subset \Sm_{\scr X}^{\mathrm{sch}} \subset \Sm_{\scr X}^{\mathrm{repr}} \subset \Stk_{/\scr X} \] the full subcategories on those stacks $\scr Y$ smooth over $\scr X$ such that whenever $A$ is an affine scheme and $A \to \scr X$ is any morphism, then $\scr Y \times_{\scr X} A$ is respectively an affine scheme, a quasi-affine scheme, a quasi-projective $A$-scheme, a scheme, or an algebraic space.
Motivic homotopy theory of $\scr X$ will be built with one of these categories as the starting point.
Depending on the choice, one may obtain theories with more or less favorable properties.
At least if $\scr X = S \sslash G$ (and $X$ separated in case of $\Sm_S^\mathrm{aff}$), then all of these choices lead to the same theory \cite[Propositions 2.11 and 2.12]{bachmann-MGM} \cite[\S A.3.4]{khan2021generalized} \cite[Remark 3.3]{Hoyois:6functor}.

In the sequel, if $\scr X = S \sslash G$, we denote by $\Sm_{\scr X}$ any of the above categories, noting that whatever construction we are performing ultimately does not depend on the choice.
If $\scr X$ is not of this form, one should take $\Sm_{\scr X} := \Sm_{\scr X}^\mathrm{qaff}$ to be consistent with \cite{bachmann-MGM}.

\subsection{Basic notions}
Let $S$ be a scheme with a $G$-action.
Put $\Sm_S^G := \Sm_{S \sslash G}$.
Recall that a \emph{family of subgroups} $\scr F$ of $G$ is a set of subgroups which is closed under conjugation and passage to subgroups.
\begin{example}
The families that are of greatest interest in this paper are the extreme ones: the family $\scr F_\all$ of all subgroups of $G$, and the family $\scr F_\triv$ consisting of only the trivial subgroup $\{e\}$.
\end{example}
Write $\Sm_S^G[\scr F] = \Sm_{S \sslash G}[\scr F]$ for the subcategory on those schemes with isotropy in $\scr F$ \cite[Definition 4.19]{bachmann-MGM}.
One puts \[ \Spc^G(S)[\scr F] = \Spc(S \sslash G)[\scr F] = L_\mot \PSh(\Sm_{S \sslash G}[\scr F]). \]
The canonical functor \[ \Spc^G(S)[\scr F] \to \Spc^G(S) \] is fully faithful with essential image generated under colimits by $\Sm_S^G[\scr F]$ \cite[Corollary 4.29]{bachmann-MGM}.

We have $T^G_B = \A^G_B/\A^G_B \setminus 0 \in \Spc^G(B)_*$, and we obtain $T^G = T^G_S \in \Spc^G(S)_*$ by base change.
We define \[ \SH^G(S) = \SH(S \sslash G) = \Spc^G(S)_*[(T^G)^{-1}] \] and denote by \[ \SH^G(S)[\scr F] \subset \SH^G(S) \] the subcategory generated under colimits by $(T^G)^{\wedge n} \wedge \Sigma^\infty_+(\Spc^G(S)[\scr F])$, for $n \in \Z$.

\begin{warning}
While we can make this definition for any $S$ and $G$, the resulting category is most well-behaved if $G$ is \emph{tame}, i.e. $|G|$ is invertible on $S$.
Indeed this assumption is needed for most non-trivial results of \cite{Hoyois:6functor}.
\end{warning}

\subsection{Models for colocalizations at $\scr F$}

\begin{definition} \label{def:eg} Let $\scr F$ be a family of subgroups of $G$. Then the \emph{geometric universal $\scr F$-space over $G$} is the presheaf on $\Sm^G_S$ defined by
\[
\EE \scr F(T) = \begin{cases}
        \emptyset, & T \not\in \Sm_S^G[\scr F] \text{ (i.e. there exists $x \in T$ such that $\stab(x) \not\in \scr F$)}\\
        * & T \in \Sm_S^G[\scr F].
        \end{cases}
\]
In particular we write $\EE G$ for $\EE \scr F_{\triv}$
\end{definition}

We have the following proposition extracted from \cite{gepner-heller} or \cite{bachmann-MGM}.
\begin{proposition}\todo{remove (3) and (4)?} \label{prop:coloc-smash} \label{prop:coloc} Suppose that is $\scr F$ a family.
\begin{enumerate}
\item Consider the adjunction induced by the inclusion $\scr F \subset \scr F_{\all}$
\[
u_!: \Spc^G(S)[\scr F] \rightleftarrows \Spc^G(S): u^*.
\]
There is canonical equivalence of endofunctors
\[
u_!u^*(-) \simeq \EE \scr F \times -.
\]
\item Consider the stabilized adjunction
\[
u_!: \SH^G(S)[\scr F] \rightleftarrows \SH^G(S): u^*.
\]
There is canonical equivalence of endofunctors
\[
u_!u^* \simeq \EE \scr F_+ \wedge -.
\]
\item The presheaf $\EE \scr F$ is motivic local.
\item The presheaf $\EE \scr F$ is represented by an ind-smooth $G$-scheme. 
\end{enumerate}
\end{proposition}
\begin{proof}
For (1) and (2) see \cite[Lemmas 4.30 and A.4]{bachmann-MGM}.
For (3) and (4) see \cite[Proposition 3.3, Example 3.5, Proposition 3.7]{gepner-heller}.
\end{proof}

\subsection{Quotients} \label{subsec:quotients}
Let $S$ be a $G$-scheme.
A quotient of $S$ by $G$ is an initial $G$-scheme under $S$ with trivial $G$-action.
Quotients are not guaranteed to exist, but this is so if the action is free and $S$ is quasi-projective \cite[Tag 07S7]{stacks}; equivalently the stack $S \sslash G$ is in fact a scheme.
When working with quotients, we thus let $\Sm^G_S$ denote $\Sm_{S \sslash G}^{\mathrm{qaff}}$ or $\Sm_{S \sslash G}^{\qproj}$.\footnote{Or we could work with $\Sm_{S \sslash G}^{\mathrm{repr}}$ and ask for algebraic space quotients.}
Still assuming that $S$ has free $G$-action, the morphism $S \to S/G$ is finite étale (since it is a $G$-torsor and $G$ is finite discrete, whence finite étale), and hence $S/G$ is smooth over some base $B$ if $S$ is.
All in all we see that there is a functor \[ (\ph)/G: \Sm_B^G[\scr F_\triv] \to \Sm_B \] which is in fact a partial left adjoint to \[ \triv: \Sm_B \to \Sm_B^G. \]

\begin{proposition} \label{prop:quot} Let $B$ have the trivial $G$-action.
There are colimit-preserving functors \[ (\ph)/G:\quad \SH^G(B)[\scr F_{\triv}] \rightarrow \SH(B) \quad\text{and}\quad \Spc^G(B)[\scr F_{\triv}] \rightarrow \Spc(B) \] such that the following hold.
\begin{enumerate}
\item The following diagram commutes
\[
\begin{tikzcd}
\Sm_{B\sslash G}^{\qproj}[\scr F_{\triv}] \ar{r}{(\ph)/G} \ar{d} & \Sm_B \ar{d} \\
\Spc^G(B)[\scr F_\triv] \ar{r}{(\ph)/G} \ar[swap,"\Sigma^{\infty}_+"]{d} & \Spc(B) \ar{d}{\Sigma^{\infty}_+} \\
\SH^G(B)[\scr F_{\triv}] \ar{r}{(-)/G} & \SH(B)
\end{tikzcd}
\]
\item The right adjoint to $(-)/G$ is given by $\SH(B) \xrightarrow{\triv} \SH^G(B) \xrightarrow{\wedge \EE G_+} \SH^G(B)[\scr F_{\triv}]$, and similarly for $\Spc(\ph)$.
\item The functors $(\ph)/G$ are oplax symmetric monoidal.
\end{enumerate}
\end{proposition}
\begin{proof} 
(1, 2) See \cite[Corollary 4.42]{bachmann-MGM}.

(3) Consequence of being left adjoint to a symmetric monoidal functor \cite[Theorem 4.5]{rune-lax}.
\end{proof}

\begin{lemma} \label{lemm:triv-quot-projection-formula}
For $E \in \SH(S)$ and $F \in \SH^G(S)$ we have $(\EE G_+ \wedge E^\triv \wedge F)/G \wequi E \wedge (\EE G_+ \wedge F)/G$.
\end{lemma}
\begin{proof}
It suffices to treat the case where $E = \Sigma^\infty_+ X$, $F = \Sigma^\infty_+ Y$ with $X \in \Sm_S$, $Y \in \Sm_{S \sslash G}[F_\triv]$.
For this we must show that \[ (X \times Y)/G \wequi X \times (Y/G), \] which is clear.
\end{proof}

\begin{definition} \label{def:gho}
We call the composite functor \[ (\ph)_{\hh G}: \SH^G(S) \xrightarrow{\wedge \EE G_+} \SH^G(S)[\scr F_{\triv}] \xrightarrow{(\ph)/G} \SH(S) \] the \emph{geometric homotopy orbits} functor.
\end{definition}

\end{subappendices}

\bibliographystyle{amsalpha}
\bibliography{powerops}

\end{document}